\documentclass[a4,12pt]{article}%
\usepackage{a4wide}
\usepackage{graphicx}
\usepackage{graphics}
\usepackage{xcolor}
\usepackage{setspace,mathrsfs}
\newcommand{\Eqref}[1]{Eq.\,\eqref{#1}}
\usepackage[maxbibnames=9]{biblatex}

\usepackage{amsmath}
\usepackage{amsfonts}
\usepackage{amssymb}
\usepackage{enumerate}
\newtheorem{theorem}{Theorem}[section]

\newtheorem{corollary}[theorem]{Corollary}

\newtheorem{definition}[theorem]{Definition}
\newtheorem{example}[theorem]{Example}

\newtheorem{comment}[theorem]{Remark}

\newenvironment{proof}[1][Proof]{\noindent\textbf{#1.} }{\ \rule{0.5em}{0.5em}}

\newcommand{\E}{{\rm \bf E}}
\renewcommand{\P}{{\rm \bf P}}

\newcommand{\len}{{\rm stage}}

\newcommand{\prob}{{\rm \bf P}}

\newcommand{\supp}{{\rm supp}}
\newcommand{\dN}{{\mathbb N}}

\newcommand{\dR}{{\mathbb R}}

\newcommand{\calB}{{\cal B}}

\newcommand{\calF}{{\cal F}}

\newcommand{\calR}{{\mathscr R}}

\newcommand{\ep}{\varepsilon}

\newcounter{figurecounter}
\begin{document}

\title{Stochastic Games with General Payoff Functions%
\thanks{
Solan acknowledges the support of the Israel Science Foundation, grants
\#217/17 and \#211/22.}}

\author{J\'{a}nos Flesch\footnote{Department of Quantitative Economics, 
Maastricht University, P.O.Box 616, 6200 MD, The Netherlands. E-mail: j.flesch@maastrichtuniversity.nl.}
\and Eilon Solan\footnote{School of Mathematical Sciences, Tel-Aviv University, Tel-Aviv, Israel, 6997800, E-mail: eilons@tauex.tau.ac.il.}}

\maketitle

\begin{abstract}
We consider multiplayer stochastic games in which the payoff of each player is a bounded and Borel-measurable function of the infinite play. By using a generalization of the technique of Martin (1998) and Maitra and Sudderth (1998), we show four different existence results. In each stochastic game, it holds for every $\ep>0$ that (i) each player has a strategy that guarantees in each subgame that this player's payoff is at least her maxmin value up to $\ep$, (ii) there exists a strategy profile under which in each subgame each player's payoff is at least her minmax value up to $\ep$, (iii) the game admits an extensive-form correlated $\ep$-equilibrium, and (iv) there exists a subgame that admits an $\ep$-equilibrium.
\end{abstract}



\section{Introduction}

Stochastic games, introduced by Shapley (1953), are dynamic games where the players' actions
affect the evolution of a state variable.
These games have been extensively studied in the past 70 years,
both under the discounted payoff
(e.g., Fink (1964), Takahashi (1964), Nowak (1985), Mertens and  Parthasarathy (1987), Duggan (2012), Levy (2013), Levy and McLennan (2015)),
and the long-run average payoff and the uniform approach
(e.g., Mertens and Neyman (1981), Vrieze and Thuijsman (1989), Vieille (2000a,b), Solan and Vieille (2002), Sorin and Vigeral (2013), Venel (2015), Renault and Ziliotto (2020)). For an overview
we refer to 
Filar and Vrieze (2012),
Solan and Vieille (2015), Ja\'{s}kiewicz and Nowak (2018), Levy and Solan (2020), 
and Solan (2022).

Stochastic games with general payoffs, introduced by Blackwell (1969), have also received attention in the literature. In these games, the payoff of a player is usually defined as a bounded and Borel-measurable function of the infinite play. Naturally, the techniques used for these payoff functions have been quite different than those employed for the discounted and the long-run average payoffs. 
Martin (1998), see also Maitra and Sudderth (1998), 
introduced a powerful technique
to the study of two-player zero-sum stochastic games with a general payoff function.
They used the determinacy of alternating-move games to
show that to each two-player zero-sum stochastic game with finite action spaces and countable state space and 
each history of the stochastic game
one can associate a certain auxiliary one-shot game, 
which has the same action spaces as the stochastic game, such that a player can play well in the stochastic game by playing well in the one-shot game at each history. 

These auxiliary one-shot games associated with the histories are induced by a single function assigning a real number to each history, which we term the \emph{Martin function}. This function has been generalized to multiplayer stochastic games in Ashkenazi-Golan, Flesch, Predtetchinski, and Solan (2022a), and they used it to prove the existence of an $\ep$-equilibrium in multiplayer repeated games with tail-measuarable payoffs. This generalization of the Martin function was further applied in Ashkenazi-Golan, Flesch, Predtetchinski, and Solan (2022b) to study regularity properties of the minmax and maxmin values and 
by using them to prove the existence of an $\ep$-equilibrium in multiplayer repeated games under some conditions on the minmax values, 
in Ashkenazi-Golan, Flesch and Solan (2022) to prove the existence of an $\ep$-equilibrium in two-player absorbing games with tail-measurable payoffs, and in Flesch and Solan (2022a) to prove the existence of an $\ep$-equilibrium in two-player stochastic
games with 
a finite state space and
shift-invariant payoffs. 

In the papers mentioned in the previous paragraph,
the use of the Martin function is hidden within the proofs,
and the 
existence 
results are proven for specific families of games:
repeated games with tail-measurable payoffs, 
repeated games under some conditions on the minmax values, or
two-player stochastic games with a finite state space and shift-invariant payoffs.
The goal of this paper is to
emphasize the significance of the Martin function to stochastic games,
and
derive four different existence results for all multiplayer stochastic games with general payoffs, each proven by the use of the Martin function.\smallskip

\noindent\textbf{[1]} We prove that each player has a \emph{subgame $\ep$-maxmin strategy}, for every $\ep>0$. This is a strategy that guarantees, regardless of the opponents' strategies, that in each subgame this player’s payoff is at least her maxmin value up to $\ep$.

This result was already proven, although not explicitly stated, in Mashiah-Yaakovi (2015). 
Unlike the proof in Mashiah-Yaakovi (2015), our proof is
straightforward, based on the Martin function. We also refer to Flesch, Herings, Maes, and Predtetchinski (2021) for the case of only two players.\smallskip

\noindent\textbf{[2]} We prove the existence of a \emph{minmax $\ep$-acceptable strategy profile}, for every $\ep>0$. This is a strategy profile under which in each subgame each player’s payoff is at least her minmax value up to $\ep$. Thus, such a strategy profile induces individually rational payoffs in all subgames, up to $\ep$.

A weaker version of the concept of minmax $\ep$-acceptable strategy profiles was defined in Solan (2018) in the context the long-run average payoff, where the expected payoffs are only required to be individually rational, up to $\ep$, from the initial state of the game. As Solan argued, the existence of such a strategy profile follows from Solan and Vieille (2002).

A priori, it is not clear that a minmax $\ep$-acceptable strategy profile always exists. Indeed, it is not easy to find techniques, other than the Martin function as in our paper, that are suited to the study of minmax $\ep$-acceptable strategy profiles. For example, while one-shot games admit minmax 0-acceptable strategy profiles, we are not aware of a proof for their existence, except by turning to the stronger notion of 0-equilibrium and using 
the fact
that each 0-equilibrium is automatically minmax 0-acceptable. In stochastic games, one cannot use a similar reasoning: while every subgame-perfect $\ep$-equilibrium would be automatically minmax $\ep$-acceptable, a subgame-perfect $\ep$-equilibrium does
not always exist; for a counter-example, see Flesch, Kuipers, Mashiah-Yaakovi, Schoenmakers, Shmaya, Solan, and Vrieze (2014). 

Simon (2016), on Page 202, raised the question whether there can be a three-player stochastic game in which the sum of the payoffs is zero for every infinite play, yet each player’s minmax value is strictly positive. Our result answers this question: such a game cannot exist.\smallskip

\noindent\textbf{[3]} We prove the existence of an \emph{extensive-form correlated $\ep$-equilibrium}, for every $\ep>0$.

This results was shown in Mashiah-Yaakovi (2015). Our proof based on the Martin function is, once again, short and straightforward.\smallskip

\noindent\textbf{[4]} We prove the existence of an \emph{$\ep$-solvable subgame}, for every $\ep>0$. That is, 
for every $\ep > 0$ there is a history such that, in the subgame defined by that history, an $\ep$-equilibrium exists.

In the specific case of the long-run average payoff, it is only the current state that matters when one considers a subgame. Therefore, for a finite state space, only finitely many essentially different subgames can arise for the long-run average payoff,
and hence the existence of an $\ep$-solvable subgame is equivalent to the existence of an initial state at which an $\ep$-equilibrium exists, for every $\ep > 0$.
Using this property, among others, the corresponding existence result for the long-run average payoff was shown in Thuijsman and Vrieze (1991) when there are only two players (they used the term \emph{easy initial state}), and in Vieille (2000c) for more than two players (who introduced the term \emph{solvable state}). 

Our proof, in addition to using the Martin function, requires techniques to detect deviations of the players from non-stationary strategies.
To this end we use a recent result by Alon, Gunby, He, Shmaya, and Solan (2022),
which was 
also 
used in Flesch and Solan (2022b) in an alternative proof 
for the existence of an $\ep$-equilibrium in multiplayer repeated games with tail-measurable payoffs.
As far as we know, the existence of $\ep$-solvable 
subgames is the first theorem 
where the use of the result of Alon, Gunby, He, Shmaya, and Solan (2022) is imperative for a proof.

Our result connects and gives a very partial answer to the long-standing open problem whether every multiplayer stochastic game with 
finite action sets, 
finite or countably infinite state space,
and bounded and Borel-measurable payoffs admits an $\ep$-equilibrium, for every $\ep>0$. As mentioned earlier, a subgame-perfect $\ep$-equilibrium does not always exist, so in some stochastic games there is no strategy profile that would induce an $\ep$-equilibrium in all subgames simultaneously.\medskip

We remark that none of our existence results [1]--[4] holds for $\ep=0$; there are various counter-examples even in the context of two-player zero-sum stochastic games with the long-run average payoffs.\medskip

\noindent\textbf{Organisation of the paper.} 
The model of stochastic games is described in Section~\ref{section:model}.
The concept of the Martin function is defined in Section~\ref{section:martin}.
The four existence results are presented and proven, by applying the Martin function, in Section~\ref{section:application}.
Section~\ref{section:discussion} concludes.

\section{The Model}
\label{section:model}

Let $\dN=\{1,2,\ldots\}$. For a nonempty finite or countably infinite
set $X$, let $\Delta(X)$ denote the set of probability distributions on $X$.

\begin{definition}
\label{def:stochastic:game}
A \emph{stochastic game}\footnotemark\,is a tuple $\Gamma = (I, S, (A_i)_{i\in I}, p, (f_i)_{i\in I})$,
where
\begin{itemize}
\item   $I$ is a nonempty finite set of players.
\item   $S$ is a nonempty finite or countably infinite set of states.
\item   $A_i$ is a nonempty finite set of actions of player $i$, for each $i \in I$. Let $A := \prod_{i \in I} A_i$ denote the set of action profiles.
\item   $p : S\times A \to \Delta(S)$ is the transition function. For states $s,s'\in S$ and action profile $a\in A$, we denote by $p(s'\mid s,a)$ the probability of state $s'$ under $p(s,a)$.\smallskip\\
Let $\mathscr{R}$ denote the set of runs, i.e., all sequences $(s^1,a^1,s^2,a^2,\ldots)\in (S\times A)^\infty$ such that $p(s^{n+1}\mid s^n,a^n)>0$ for each $n\in\dN$. 
We endow the set $(S\times A)^\infty$ with the product topology, where the finite sets $S$ and $A$ have their natural discrete topologies, and then we endow the set $\mathscr{R}$ of runs, which is a closed subset of $(S\times A)^\infty$, with the subspace topology. We denote by $\mathcal{B} (\mathscr{R})$ the corresponding Borel sigma-algebra on $\mathscr{R}$. 
\item   $f_i : \mathscr{R} \to \dR$ is a bounded and Borel-measurable payoff function for player~$i$, for each $i \in I$.
\end{itemize}
\end{definition}
\footnotetext{Since the payoff function is not derived from stage payoffs, this model is also called a \emph{multi-stage stochastic game}.}

The game is played in stages in $\dN$. 
The play starts in a given initial state $s^1 \in S$.
In each stage $n \in \dN$ the play is in some state $s^n\in S$
and the players simultaneously choose actions; denote by $a_i^n$ the action selected by player~$i$. 
This induces an action profile $a^n=(a_i^n)_{i\in I}$, which is observed by all players. Then, the state $s^{n+1}$ for stage $n+1$ is drawn from the distribution $p(\cdot\mid s^n,a^n)$ and is observed by all players.

\begin{comment}\rm
In our model, the action sets $(A_i)_{i \in I}$ are independent of the state. 
This assumption is used to simplify the exposition, and all the statements and proofs in the paper can be extended to stochastic games in which the action sets are finite yet depend on the state.
\end{comment}

\begin{comment}\label{undisc}\rm 
In the literature, several payoff functions are derived from rewards that the players receive in each stage of the game. One of the most studied such payoff function is the long-run average payoff. Given a function $z_i : S \times A \to \dR$ for player $i$ that assigns a reward $z_i(s,a)$ to each state $s \in S$ and each action profile $a\in A$, player $i$'s long-run average payoff is defined as
\[ f_i(s^1,a^1,\dots)\, :=\, \limsup_{n \to \infty} \frac{1}{n}\sum_{k=1}^n z_i(s^k,a^k). \]
\end{comment}\medskip

\noindent\textbf{Histories.} A \emph{history} in stage $n\in\dN$ is a sequence $(s^1,a^1,\ldots,s^{n-1},a^{n-1},s^n)\in (S\times A)^{n-1}\times S$ such that $p(s^{k+1}\mid s^k,a^k)>0$ for each $k=1,\ldots,n-1$. The set of histories in stage $n$ is denoted by $H^n$ and the set of histories is denoted by $H := \bigcup_{n=1}^\infty H^n$.

The current stage (or length) of a history $h\in H^n$ is denoted  by $\len(h):=n$ and the final state of $h$ is denoted by $s_h$. 
For two histories $h,h'\in H$, we write $h\preceq h'$ if $h'$ extends $h$ (with possibly $h=h'$), and we write $h\prec h'$ if $h\preceq h'$ and $h\neq h'$. If the final state $s_h$ of $h$ coincides with the first state of $h'$, then we write $hh'$ for the concatenation of $h$ with $h'$. Similarly, for a history $h\in H$ and a run $r\in\mathscr{R}$, we write $h\prec r$ if $r$ extends $h$, and if the final state $s_h$ of $h$ coincides with the first state of $r$, then we write $hr$ for the concatenation of $h$ with $r$.

For every run $r = (s^1,a^1,s^2,a^2,\ldots)\in\mathscr{R}$ and every stage $n \in \dN$,
we denote by $r^n := (s^1,a^1,s^2,a^2,\ldots,s^n) \in H$ the prefix of $r$ in stage $n$.
\medskip

\noindent\textbf{Subgames.} 
Each history induces a \emph{subgame} of $\Gamma$. Given $h\in H$, the subgame that starts at $h$ is the game $\Gamma_h = (I,S,(A_i)_{i \in I},p, (f_{i,h})_{i \in I})$ having $s_h$ as the initial state, where $f_{i,h}(r) := f_i(hr)$ for each run $r\in\mathscr{R}$ that starts in state $s_h$.\medskip

\noindent\textbf{Mixed actions.}  
A \emph{mixed action} for player $i\in I$ is a probability distribution $x_i$ on $A_i$. The set of mixed actions for player $i$ in state $s$ is thus $\Delta(A_i)$.  
The probability that $x_i$ places on action $a_i\in A_i$ is denoted by $x_i(a_i)$. 

A \emph{mixed action profile} is a collection $x=(x_i)_{i\in I}$ of mixed actions, one for each player. For a player $i\in I$, a mixed action profile of her opponents
is a collection $x_{-i}=(x_j)_{j\in I\setminus\{i\}}$ of mixed actions. 

The \emph{support} of the mixed action $x_i$ is $\supp(x_i):=\{a_i\in A_i\colon x_i(a_i)>0\}$. The support of a mixed action profile $x=(x_i)_{i\in I}$ is $\supp(x):=\prod_{i \in I} \supp(x_i)\subseteq A$.

For a mixed action profile $x=(x_i)_{i\in I}$, we denote the probability of moving from state $s$ to state $s'$ under $x$ by
\[p(s'\mid s,x)\,:=\,\sum_{a=(a_i)_{i\in I}\in A}\,\left( p(s'\mid s,a)\cdot\prod_{i\in I}\, x_i(a_i)\right).\]

\noindent\textbf{Strategies.}  
A (behavior) \emph{strategy} of player~$i$ is a function $\sigma_i : H \to \Delta(A_i)$. We denote by $\sigma_i(a_i \mid h)$ the probability on the action $a_i\in A_i$
under $\sigma_i(h)$. The interpretation of $\sigma_i$ is that if history $h$ arises, then $\sigma_i$ recommends to select an action according to the mixed action  $\sigma_i(h)$. We denote by $\Sigma_i$ the set of strategies of player~$i$.

The \emph{continuation of a strategy $\sigma_i$} in the subgame that starts at a history $h\in H$ is denoted by $\sigma_{i,h}$; 
this is a function that maps each history $h'$ having $s_h$ as its first state to the mixed action $\sigma_{i,h}(h'):=\sigma_i(hh')\in\Delta(A_i)$.

A \emph{strategy profile} is a collection $\sigma=(\sigma_i)_{i\in I}$ of strategies, one for each player. We denote by $\Sigma$ the set of strategy profiles. For a player $i\in I$, the strategy profile of her opponents is a collection $\sigma_{-i}=(\sigma_j)_{j\in I\setminus \{i\}}$ of strategies. We denote by $\Sigma_{-i}$ the set of strategy profiles of player $i$'s opponents.\medskip

\noindent\textbf{Expected payoffs.}
By Kolmogorov's extension theorem, each strategy profile $\sigma$ together with an initial state $s$ induces a unique probability measure $\P_{s,\sigma}$ on $(\mathscr{R},\mathcal{B} (\mathscr{R}))$. The corresponding expectation operator is denoted by $\E_{s,\sigma}$. 
Player $i$'s \emph{expected payoff} under the strategy profile $\sigma$ is
\[\E_{s,\sigma}[f_i]\,=\,\int_{r\in\mathscr{R}}f_i(r)\ \P_{s,\sigma}(dr).\]

Given a history $h\in H$, 
in the subgame $\Gamma_h$, each strategy profile $\sigma$ similarly induces a unique probability measure $\P_{h,\sigma}$ on $(\mathscr{R},\mathcal{B} (\mathscr{R}))$.
The corresponding expectation operator is denoted by $\E_{h,\sigma}$.  Player $i$'s expected payoff under strategy profile $\sigma$ in $\Gamma_h$ is
\[\E_{h,\sigma}[f_i]\,=\,\int_{r\in\mathscr{R}}f_{i,h}(r)\ \prob_{h,\sigma}(dr).\]

\noindent\textbf{Minmax value and maxmin value.}
The \emph{minmax value} of player $i\in I$ for the initial state $s\in S$ is the quantity
\begin{equation}
\label{def-minmax}
\overline v_i(s)\,:=\,\inf_{\sigma_{-i}\in\Sigma_{-i}}\sup_{\sigma_i\in\Sigma_i}\E_{s,\sigma_i,\sigma_{-i}}[f_i].
\end{equation}
Intuitively, $\overline v_i(s)$ is the highest payoff that player $i$ can defend against any strategy profile of her opponents. The minmax value $\overline v_i(h)$ of player $i$ in the subgame that starts at history $h$ is defined analogously.

The \emph{maxmin value} of player $i\in I$ for the initial state $s\in S$ is the quantity
\begin{equation}
\label{def-maxmin}
\underline v_i(s)\,:=\,\sup_{\sigma_i\in\Sigma_i}\inf_{\sigma_{-i}\in\Sigma_{-i}}\E_{s,\sigma_i,\sigma_{-i}}[f_i].
\end{equation}
Intuitively, $\underline v_i(s)$ is the highest payoff that player $i$ can guarantee to receive regardless of the strategy profile of her opponents. The maxmin value $\underline v_i(h)$ of player $i$ in the subgame that starts at history $h$ is defined analogously. 
Note that $\overline v_i(s)\geq\underline v_i(s)$ for each player $i\in I$ and each state $s\in S$, and analogous inequalities hold in the subgame at each history $h\in H$.
\medskip

\noindent\textbf{Equilibrium.}
Let $\ep\geq 0$. A strategy profile $\sigma^*$ is called an \emph{$\ep$-equilibrium for the initial state $s\in S$}, if we have $\E_{s,\sigma^*}[f_i] \,\geq\, \E_{s,\sigma_i,\sigma^*_{-i}}[f_i] - \ep$ for each player $i\in I$ and each strategy $\sigma_i \in \Sigma_i$. It follows from the definitions that if $\sigma^*$ is an $\ep$-equilibrium for the initial state $s\in S$, then $\E_{s,\sigma^*}[f_i]\geq \overline v_i(s)-\ep$ for each player $i\in I$. A strategy profile $\sigma$ is called an \emph{$\ep$-equilibrium} if it is an $\ep$-equilibrium for each initial state $s\in S$.

\section{Martin's Function}
\label{section:martin}

Martin (1998) and Maitra and Sudderth (1998) showed that, in every two-player zero-sum stochastic game with a bounded and Borel-measurable payoff function, it is possible to assign an auxiliary one-shot zero-sum game to each history, with the action sets $A_1$ and $A_2$ for the players, in such a way that if at all histories a player plays well in the corresponding one-shot game, then she plays well in the stochastic game too. This is a powerful result, as it shows how to play well in the zero-sum stochastic game, by decomposing the infinite duration game into suitable one-shot games. The purpose of this section is to extend this result to non-zero-sum stochastic games, with possibly more than two players.\medskip

\noindent\textbf{The auxiliary one-shot games.} Let $\Gamma = (I,S,(A_i)_{i \in I}, p,(f_i)_{i \in I})$ be a stochastic game. Suppose that we are given a function $D=(D_i)_{i\in I}:H\to \dR^{|I|}$. That is, at each history $h\in H$, the function $D$ specifies a number $D_i(h)$ for each player $i\in I$.

The function $D$ induces the following one-shot\footnote{The subscript $O$ is intended to remind the reader that the game under consideration is a one-shot game.} game $G_O(D,h)$ at each history $h\in H$. 
The action set of each player $i\in I$ is $A_i$, and the payoff of each player~$i\in I$ under each action profile $a\in A$ is equal to the expectation of $D_i$ at the next history:
\[\E[D_i\mid h,a]\,:=\,\sum_{s\in S}\bigl(p(s\mid s_h,a)\cdot D_i(h,a,s)\bigr).\] 
For each mixed action profile $x\in\prod_{i\in I}\Delta(A_i)$, we denote by $\E[D_i\mid h,x]$ the expectation of player $i$'s payoff under $x$:
\[\E[D_i\mid h,x]\,:=\,\sum_{a=(a_j)_{j\in I}\in A}\,\left(\E[D_i\mid h,a]\cdot\prod_{j\in I}x_j(a_j)\right).\]
In the one-shot game $G_O(D,h)$, player $i$'s \emph{minmax value} is the quantity
\[ \overline v_{O,i}(D,h) \,:=\, \min_{x_{-i} \in \prod_{j \neq i} \Delta(A_j)} \max_{x_i \in \Delta(A_i)} \E[D_i\mid h,x_i,x_{-i}], \]
and player~$i$'s \emph{maxmin value} is the quantity
\begin{equation}\label{maxmin-oneshot}
\underline v_{O,i}(D,h) \,:=\, \max_{x_i \in \Delta(A_i)} \min_{x_{-i} \in \prod_{j \neq i} \Delta(A_j)} \E[D_i\mid h,x_i,x_{-i}].
\end{equation}
Note that $\overline v_{O,i}(D,h)$ and $\underline v_{O,i}(D,h)$ are independent of $D_j$ with $j\neq i$. Therefore, sometimes we will only define $D_i$ and write $\overline v_{O,i}(D_i,h)$ and $\underline v_{O,i}(D_i,h)$.\medskip

We now state the main theorem of the section. 

\begin{theorem}\label{theorem:martin}
Let $\ep > 0$.
For each player $i \in I$
there is a bounded function $D^\ep_i : H \to \dR$, called a \emph{Martin function} for the parameter $\ep$ and player $i$, with the following properties:
\begin{enumerate}
\item $\overline v_i(h) - \ep \leq D^\ep_i(h)\leq \overline v_i(h)$, for every $h \in H$.
\item $D^\ep_i(h)\leq\overline v_{O,i}(D_i^\ep,h)$, for every $h \in H$.
\item Let $h'\in H$ be a history. If a strategy profile $\sigma \in \Sigma$ satisfies
\begin{equation}\label{submar}
\overline v_{O,i}(D^\ep_i,h)\,\leq\, \E[D^\ep_i\mid h,\sigma(h)],\ \quad\forall h \in H\text{ with }h'\preceq h,
\end{equation}
then
\[\overline v_{O,i}(D^\ep_i,h)\,\leq\, \E_{h,\sigma}[f_i],\ \quad\forall h \in H\text{ with }h'\preceq h. \]
\end{enumerate}
The result also holds when all references to the minmax value 
(both of the one-shot game and the stochastic game)
are replaced by the maxmin value.
\end{theorem}

Theorem~\ref{theorem:martin} states that, for each $\ep>0$ and each player $i$, there is a Martin function $D^\ep_i$ that assigns a real number $D^\ep_i(h)$ to each history $h$ with the following properties: 
(1) $D^\ep_i(h)$ is not more than player $i$'s minmax value $\overline{v}_i(h)$ at history $h$, 
and not less than this quantity up to $\ep$,
(2) the minmax value $\overline v_{O,i}(D_i^\ep,h)$ of the one-shot game $G_O(D^\ep_i,h)$, which is induced by $D_i^\ep$ at history $h$, is at least $D^\ep_i(h)$, 
and (3) 
any strategy profile $\sigma$ that is good locally, is also good globally; namely, if for every $h$ that extends $h'$, 
in the one-shot game $G_O(D^\ep_i,h)$ the mixed action profile $\sigma(h)$ yields to player~$i$ an expected payoff of at least $\overline v_{O,i}(D_i^\ep,h)$, then, in the stochastic game, for each history $h$ extending $h'$ player $i$'s expected payoff in the subgame at $h$ is also at least $\overline v_{O,i}(D^\ep_i,h)$. This last property is remarkably useful: if a player plays well in the one-shot game at each history, then she plays well in the stochastic game too.

Note that if we set $D_i^\ep(h) = \overline v_i(h)$,
then properties~(1) and (2) hold, yet property~(3) does not necessarily hold.
To ensure that property~(3) holds as well, we have to set $D_i^\ep(h)$ slightly lower than $\overline v_i(h)$,
as expressed by property~(1). 

Martin (1998) proved the existence of a function $D_i^\ep$ that satisfies the right-hand-side inequality in property (1) and satisfies properties~(2) and (3) of Theorem~\ref{theorem:martin} in two-player games 
(where the minmax value and the maxmin value coincide)
with a single state: $|S|=1$. Ashkenazi-Golan, Flesch, Predtetchinski, and Solan (2022a) showed how to extend the same properties to multiplayer games with a single state. Indeed, Theorem 3.6 of their paper claims properties (2) and (3) explicitly (they are called conditions (M.1) and (M.3) respectively), whereas Claim 3.5 of their paper implies the right-hand-side inequality of property (1). The left-hand-side inequality in property (1) follows by a slight modification in the proof of Theorem 3.6 of their paper: in the proof, the event called ``re-initiation'' should take direct effect, and not only from the next stage.\footnote{More precisely, following the notation in the proof of Theorem 3.6: The definition of the function $d$ needs to be changed as follows. When $h$ is a history beyond stage 1, then set $d(h)=\widehat D_{\alpha(h)}(h)$. This definition ensures that the value of $d$ cannot drop below the minmax value minus $\delta$.} By doing so, we obtain properties (1)--(3) for multiplayer games with a single state. The extension to any finite or countably infinite state space can be done in the same way as how Maitra and Sudderth (1998) extended the result of Martin (1998) to such a state space.

\begin{example}\rm 
\label{example:mn}
Suppose that $|I|=2$, and the payoff function is the long-run average payoff, cf.~Remark \ref{undisc}.
In their study of the uniform value in two-player zero-sum stochastic games,
Mertens and Neyman (1981) constructed for each $\ep > 0$ and $i \in I$ a function $D^\ep_i$ that satisfies the conditions of Theorem~\ref{theorem:martin} --
this is the function that they denote by $Y_i$ on Page 56 of their paper.

When $|I| > 2$ yet the payoff function is still the long-run average payoff, 
Neyman (2003) provided an analogous construction for the function $D^\ep_i$, which he denotes by $Y_k$ on Page 184 of his paper. $\Diamond$
\end{example}

\section{Applications}
\label{section:application}

In this section we present four applications of Theorem~\ref{theorem:martin}.
In Section~\ref{section:minmax} we show that each player has a subgame $\ep$-maxmin strategy, for every $\ep>0$, providing a short proof for a result due to Mashiah-Yaakovi (2015), see also Flesch, Herings, Maes, and Predtetchinski (2021) for the case of two-player games.
In Section~\ref{section:acceptable} we show that
there is an $\ep$-acceptable strategy profile, 
for every $\ep>0$,
extending an implication of Solan and Vieille (2002) to general Borel-measurable payoff functions.
In Section~\ref{section:extensive} we show the existence of an extensive-form correlated $\ep$-equilibrium,
for every $\ep>0$,
providing a short proof for a result due to Mashiah-Yaakovi (2015),
which itself extends a result of Solan and Vieille (2002) to general Borel-measurable payoff functions.
Finally, in Section~\ref{section:easy} we show that for every $\ep > 0$ there is a subgame of the stochastic game in which there is an $\ep$-equilibrium,
extending a result of Vieille (2000c) to general Borel-measurable payoff functions.

\subsection{ Subgame $\ep$-Maxmin Strategies}
\label{section:minmax}

A strategy of player~$i\in I$ is called subgame $\ep$-maxmin
if in each subgame it guarantees that player~$i$'s expected payoff is at least her maxmin value up to $\ep$.

\begin{definition} Let $\ep\geq 0$. 
A strategy $\sigma^*_i \in \Sigma_i$ for player $i\in I$ is 
\emph{subgame $\ep$-maxmin} if for every history $h \in H$ and every strategy profile $\sigma_{-i}\in\Sigma_{-i}$ 
\[ \E_{h,\sigma^*_i,\sigma_{-i}}[f_i] \,\geq\, \underline v_i(h) - \ep. \]
\end{definition}

The following theorem, although not explicitly stated,
is proven in Mashiah-Yaakovi (2015),
see also Flesch, Herings, Maes, and Predtetchinski (2021) for the case of only two players.

\begin{theorem}
In every stochastic game, for every $\ep > 0$, every player $i \in I$ has a subgame $\ep$-maxmin strategy.
\end{theorem}

\begin{proof}
Fix $\ep > 0$ and $i \in I$.
Consider the version of Theorem~\ref{theorem:martin} for the maxmin value,
and let $D^\ep_i$ be the function for player $i$ given by this variation.
By Eq.~\eqref{maxmin-oneshot}, for each history $h \in H$, player $i$ has a mixed action $x_i(h)$ in the one-shot game $G_O(D_i^\ep,h)$ such that for every mixed action profile $x_{-i}$ of player $i$'s opponents we have
\[\E[D^\ep_i\mid h,x_i(h),x_{-i}]\,\geq\,\underline v_{O,i}(D^\ep_i,h).\]
Define a strategy $\sigma^*_i \in \Sigma_i$ for player $i$ by
letting $\sigma^*_i(h) = x_i(h)$ for each history $h\in H$. 
We will prove that $\sigma^*_i$ is subgame $\ep$-maxmin.

To this end, fix a strategy profile $\sigma_{-i} \in \Sigma_{-i}$ of player $i$'s opponents.
By the choice of $\sigma^*_i$, for every history $h\in H$,
\[ 
\E[D^\ep_i\mid h,\sigma^*_i(h),\sigma_{-i}(h)]
\,=\, 
\E[D^\ep_i\mid h,x_i(h),\sigma_{-i}(h)] \,\geq\, 
\underline v_{O,i}(D^\ep_i,h). \]
Thus, by properties~(3), (2), and (1) of Theorem~\ref{theorem:martin}, for every history $h\in H$,
\[ 
\E_{h,\sigma^*_i,\sigma_{-i}}[f_i] \,\geq\, \underline v_{O,i}(D^\ep_i,h)
\,\geq\,D^\ep_i(h)\,\geq\,\underline v_i(h) - \ep.\]
Hence, $\sigma^*_i$ is subgame $\ep$-maxmin, as claimed.
\end{proof}

\subsection{ Minmax $\ep$-Acceptable Strategy Profiles}
\label{section:acceptable}

A minimal requirement from a reasonable strategy profile 
is that every player obtains, up to a small error-term, an expected payoff of at least her minmax value. Indeed, such a strategy profile then induces, up to a small error-term, individually rational payoffs to the players. In the context of stochastic games with the long-run average payoff, Solan (2018) proved that such a strategy profile exists by applying the results of Solan and Vieille (2002).

In this section, we consider a stronger version of this concept, where this minimal condition is required to hold in all subgames.

\begin{definition}
A strategy profile $\sigma^* \in \Sigma$ is \emph{minmax $\ep$-acceptable} if for every player $i\in I$ and every history $h\in H$
\[ \E_{h,\sigma^*}[f_i] \,\geq\, \overline v_i(h) - \ep. \]
\end{definition}

A priori, it is not clear whether every stochastic game admits a minmax $\ep$-acceptable strategy profile. Indeed, as discussed in the Introduction, every subgame-perfect $\ep$-equilibrium in the game is automatically minmax $\ep$-acceptable, but a subgame-perfect $\ep$-equilibrium does not always exist, as was shown by a counter-example in Flesch, Kuipers, Mashiah-Yaakovi, Schoenmakers, Shmaya, Solan, and Vrieze (2014). 

\begin{theorem}
\label{theorem:acceptable}
Let $\ep>0$. For each player $i \in I$, let $D_i^\ep$ be a Martin function as in Theorem~\ref{theorem:martin} (for the version with the minmax value). Let $D^\ep=(D^\ep_i)_{i\in I}$.
For each history $h \in H$, let $x(h)\in \prod_{i \in I} \Delta(A_i)$ be an equilibrium in the one-shot game
$G_O(D^\ep,h)$. 

Define a strategy profile $\sigma^*$ by letting $\sigma^*(h)=x(h)$ for each history $h\in H$. Then, the strategy profile $\sigma^*$ is minmax $\ep$-acceptable.

Consequently, in every stochastic game, for every $\ep>0$, there exists a minmax $\ep$-acceptable strategy profile.
\end{theorem}

\begin{proof}
Since $\sigma^*(h)=x(h)$ is an equilibrium in the one-shot game
$G_O(D^\ep,h)$ for each history $h\in H$, we have for each history $h \in H$ and each player $i\in I$
\[ \E[D^\ep_i\mid h,\sigma^*(h)] \,\geq\, \overline v_{O,i}(D^\ep_i,h). \]
Hence, for each history $h\in H$ and each player $i\in I$,
\begin{equation}\label{lowerpayoff}
\E_{h,\sigma^*}[f_i] \,\geq\, \overline v_{O,i}(D_i^\ep,h) \,\geq\, D^\ep_i(h)\,\geq\,\overline v_i(h) - \ep,
\end{equation}
where the inequalities hold respectively by properties~(3), (2), and (1) of Theorem~\ref{theorem:martin}.
\end{proof}

\subsection{ Extensive-Form Correlated $\ep$-Equilibria}
\label{section:extensive}

An extensive-form correlated $\ep$-equilibrium is an $\ep$-equilibrium in an extended game, 
which includes a mediator, who sends a private message to each player at every stage.
Solan and Vieille (2002) proved the existence of an extensive-form correlated $\ep$-equilibrium, for every $\ep>0$, 
in all stochastic games with  finitely many states and
the long-run average payoff.
Mashiah-Yaakovi (2015) extended this result to all stochastic games with 
countable many states and
payoff functions that are bounded and Borel-measurable.
In this section we show how 
the result of Mashiah-Yaakovi (2015) 
follows from Theorem~\ref{theorem:martin}.

\begin{definition}
A stochastic game with a mediator is a triple $\Gamma^{M,\mu}=(\Gamma,(M_i)_{i\in I},(\mu_i)_{i\in I})$, where
\begin{itemize}
    \item $\Gamma = (I, S, (A_i)_{i\in I}, p, (f_i)_{i\in I})$ is a stochastic game as in Definition \ref{def:stochastic:game}.
    \item $M_i$ is a nonempty finite\footnote{A finite set of messages will suffice for our construction, so we do not have to consider measurable sets of messages.} set of \emph{messages} for player $i\in I$, for each $i\in I$. \smallskip\\
    Let $M := \prod_{i \in I} M_i$ denote the set of message profiles, let
\[HM \,:=\, \bigcup_{n \in \dN} \left(S \times (M \times A \times S)^{n-1}\right)\]
denote the set of \emph{histories for the mediator}, and
for each $i \in I$ let  
\[ HM_i \,:=\, \bigcup_{n \in \dN} \left(S \times (M_i \times A \times S)^{n-1}\times M_i\right) \]
denote the set of \emph{private histories for player~$i$}.
    \item $\mu_i:HM\to\Delta(M_i)$ is a function, for each $i\in I$. The collection $\mu=(\mu_i)_{i\in I}$ is called the \emph{(strategy of the) mediator}.\footnotemark
\end{itemize}
\end{definition}
\footnotetext{Our definition assumes that the signals to the players are conditionally independent given the history. 
In principle, the mediator can correlate the signals she sends to the players at each history.
Since we will not need such a correlation, we disregard it for the sake of clarity.}

The interpretation of a mediator is as follows:
In each stage $n \in \dN$, given the past history of play $(s^1,a^1,s^2,a^2,\dots,s^n)$ and given the past messages $m^1,m^2,\dots,m^{n-1}$ that the mediator already sent to the players,
the mediator randomly selects by using $\mu_i$ a private message $m^n_i$ to each player $i \in I$,
and sends it to that player.

A  (behavior) \emph{strategy of player~$i$ in $\Gamma^{M,\mu}$}
is a function $\tau_i : HM_i \to \Delta(A_i)$. Let $\mathcal{T}_i$ denote the set of strategies for player $i\in I$ in $\Gamma^{M,\mu}$. A strategy profile $\tau = (\tau_i)_{i \in I}$ in $\Gamma^{M,\mu}$ 
and a history $h \in HM$ induce a probability distribution $\prob_{h,\mu,\tau}$
on the space 
\[HM^\infty\,:=\,S \times (M \times A \times S)^\infty.\]
This is the probability distribution induced by $\tau$ and $\mu$ in the subgame of $\Gamma^{M,\mu}$
that starts at $h$.
Denote by $\E_{h,\mu,\tau}[\cdot]$ the corresponding expectation operator.

\begin{definition} Let $\ep\geq 0$. In a stochastic game with a mediator $\Gamma^{M,\mu}$, a strategy profile $\tau^*$ is an \emph{$\ep$-equilibrium} if
\[ \E_{s^1,\mu,\tau^*}[f_i] \,\geq\, \E_{s^1,\mu,\tau_i,\tau^*_{-i}}[f_i]-\ep,
\qquad\forall s^1 \in S,\, \forall i \in I,\, \forall \tau_i \in \mathcal{T}_i. \]
In a stochastic game $\Gamma$, an \emph{extensive-form correlated $\ep$-equilibrium}\footnotemark\  
is 
a triple $(M,\mu,\tau^*)$
where $\Gamma^{M,\mu}$ is a stochastic game with a mediator
and $\tau^*$ is an $\ep$-equilibrium 
in $\Gamma^{M,\mu}$.
\end{definition}
\footnotetext{The concept of extensive-form correlated $\ep$-equilibrium \emph{payoff} was defined and studied in the context of stochastic games by
Solan (2001) and Solan and Vieille (2002). 
We chose the definition provided here for simplicity.}

\begin{theorem}
In every stochastic game, for every $\ep > 0$, there exists an extensive-form correlated $\ep$-equilibrium.
\end{theorem}

\begin{proof}
The idea of the proof is as follows.
The players are supposed to follow an $(\ep/2)$-acceptable strategy profile.
To ensure that no player deviates, the mediator performs the lotteries for the players, and informs each player at every stage the action that was chosen for her.
In addition, the mediator reveals to \emph{all} players the actions she selected to everyone in the previous stage.
This mechanism ensures that a deviation is detected immediately and can be punished at the minmax level.

We turn to the formal proof.
Fix $\ep > 0$ and set $\delta := \ep/2$. 
For each player $i\in I$, let $D_i^\delta$ be the function given by Theorem~\ref{theorem:martin} (for the minmax value),
and let $D^\delta=(D_i^\delta)_{i\in I}$. For each $h \in H$,
let $x(h) \in \prod_{i \in I} \Delta(A_i)$ be an equilibrium in the one-shot game $G_O(D^\delta,h)$.

For each player $i \in I$, let $M_i := A \times A_i$.
Thus, the message sent to each player~$i$ at every stage will be a pair consisting of an action profile and an action for player~$i$. 

Now we define $\mu=(\mu_i)_{i\in I}$. Suppose that the current history is \[\widetilde h\,=\,(s^1,m^1,a^1,s^2,m^2,a^2,\dots,s^n) \in HM.\] 
At this history $\widetilde h$, the mediator randomly selects for each player $i\in I$ an action $\widehat a^n_i\in A_i$ according to the mixed action $x_i(s^1,a^1,s^2,a^2,\dots,s^n)$. Then, the mediator sends to player~$i$ the message $m^n_i=((\widehat a^{n-1}_j)_{j\in I},\widehat a^n_i)\in M_i$; for stage $n=1$ the first coordinate is irrelevant and is just some fixed action profile $\widehat a^0\in A$. The interpretation of $\widehat a^n_i$ is that the mediator recommends player~$i$ to play action $\widehat a^n_i$. That is, the message $m^n_i$ to player~$i$ consists of the actions that were recommended in the previous stage, and a recommended action for player~$i$ in the current stage. Formally, we define for each $i\in I$ the function 
$\mu_i : HM \to \Delta(M_i)$ as follows:
given a history $\widetilde h = (s^1,m^1,a^1,s^2,m^2,a^2,\dots,s^n) \in HM$ in $\Gamma^{M,\mu}$,
and denoting $m_i^{n-1} = \bigl( (\widehat a_j^{n-1})_{j \in I}, \widehat a_i^n\bigr)$ and $\widehat a^{n-1} := (\widehat a_j^{n-1})_{j \in I}$,
we let
\[ \mu_i(\widetilde h) \,:=\, \mathbf{1}_{\widehat a^{n-1}} \otimes x_i(s^1,a^1,s^2,a^2,\dots,s^n). \]

At the private history $\widetilde h_i = (s^1,m^1_i,a^1,s^2,m^2_i,a^2,\dots,s^n,m^n_i) \in HM_i$, if
$a^k_j = \widehat a^k_j$ holds for every player $j \in I$ and every stage $k \in \{1,2,\dots,n-1\}$, 
then all players followed the actions recommended by the mediator to them.
If this condition does not hold, 
denote by $k^*$ the first stage in which some player 
did not follow the action recommended by the mediator,
and by $i^*$ the minimal index of the player who did not follow the recommendation in stage $k^*$.
Denote by $h^* := (s^1,a^1,\dots,s^{k^*},a^{k^*},s^{k^*+1})\in H$ the history in the stage after the deviation occurs.
Note that $k^*$, $i^*$, and $h^*$ are all random variables that depend on the play of the game.

For each player $i\in I$, let $\tau^*_i : HM_i \to \Delta(A_i)$ be the strategy in the game with mediator that 
follows the recommendation of the mediator, unless some player deviates, whereupon the deviator is punished 
at her minmax value. That is:
\begin{itemize}
\item For every private history $\widetilde h_i = (s^1,m^1_i,a^1,s^2,m^2_i,a^2,\dots,s^n,m^n_i) \in HM_i$ along which no deviation from the recommendation of the mediator was made, $\tau^*_i(\widetilde h_i)$ follows the recommendation of the mediator: $\tau^*_i(\widetilde h_i)$ places probability 1 on the action $\widehat a^n_i$ where $m^n_i=((\widehat a^{n-1}_j)_{j\in I},\widehat a^n_i)$.
\item
Once a private history $\widetilde h_i = (s^1,m^1_i,a^1,s^2,m^2_i,a^2,\dots,s^n,m^n_i) \in HM_i$ occurs in which, based on the message $m^n_i$, player $i$ (and the other players too) notices a deviation in stage $k^*=n-1$ from the recommendation of the mediator, then all players $i \neq i^*$ switch to a punishment strategy profile against player~$i^*$,
namely,
a strategy profile that lowers player~$i^*$'s payoff to $\overline v_{i^*}(s^1,a^1,s^2,\dots,s^n) +\delta$.
\end{itemize}

We argue that $\tau^* = (\tau^*_i)_{i \in I}$ is an $\ep$-equilibrium in the game with mediator 
$\Gamma^{M,\mu}$. Notice that when all players follow their recommendations, namely, they adopt the strategy profile $\tau^*$,
the players in fact implement the minmax $\delta$-acceptable strategy profile $\sigma^*$ given in Theorem \ref{theorem:acceptable}. 

Suppose that some player, say player $i\in I$, considers deviating for the first time at the private history $\widetilde h_i = (s^1,m^1_i,a^1,s^2,m^2_i,a^2,\dots,s^n,m^n_i) \in HM_i$. The corresponding history in the stochastic game is then $h=(s^1,a^1,s^2,a^2,\dots,s^n)$. If player $i$ decides not to deviate, then the strategy profile $\sigma^*$ will be implemented, and hence player $i$'s expected payoff will be $\E_{h,\sigma^*}[f_i]$. 
By \Eqref{lowerpayoff},
\begin{equation}\label{ath}
\E_{h,\sigma^*}[f_i]\,\geq\,D^\delta_i(h).
\end{equation}

Suppose now that player~$i$ deviates at $\widetilde h_i$ from the mediator's recommendation and selects some other action, say $a_i\in A_i$. According to $\tau^*$, from the following stage and on, she will be punished at her minmax value plus $\delta$. That is, her payoff will be at most
\begin{eqnarray*}
\E_{a_i,\sigma^*_{-i}(h)}[\overline v_i(h,a_i,a_{-i})] + \delta
&\leq& \E_{a_i,\sigma^*_{-i}(h)}[D^\delta_i(h,a_i,a_{-i})] + 2\delta\\
&\leq& \E_{\sigma^*(h)}[D^\delta_i(h,a_i,a_{-i})] + 2\delta\\
&\leq&
\E_{h,\sigma^*}[f_i] + 2\delta,
\end{eqnarray*}
where the first inequality holds by property (1) of Theorem~\ref{theorem:martin};
the second inequality holds because the mixed action $\sigma^*(h)$ is an equilibrium of the one-shot game $G_O(D^\delta,h)$; and the third inequality holds because $\sigma^*$ chooses the mixed action $\sigma^*(h)$ at history $h$ and, if the action profile $(a_i,a_{-i})$ is chosen at history $h$, then from the history $(h,a_i,a_{-i})$ in the next period, $\sigma^*$ gives a payoff of at least $D^\delta_i(h,a_i,a_{-i})$ by \Eqref{ath}.

Thus, the deviation can improve player $i$'s payoff by at most $ \ep= 2\delta$. Hence, $\tau^*$ is indeed an $\ep$-equilibrium in the game with mediator 
$\Gamma^{M,\mu}$.
\end{proof}

\subsection{Solvable Subgames}
\label{section:easy}

A state $s \in S$ is called 
\emph{solvable}
(or \emph{easy})
if for every $\ep > 0$, the game has an $\ep$-equilibrium when the initial state is $s$.
Thuijsman and Vrieze (1991) proved that in every two-player non-zero-sum stochastic game with 
finitely many states and
the long-run average payoff there is an easy initial state.
This result has been extended by Vieille (2000c) to multiplayer stochastic games with 
finitely many states and 
the long-run average payoff.
In this section, 
we weaken the concept of easy initial state,
and define the concept of $\ep$-solvable 
subgame,
which is a subgame that admits an $\ep$-equilibrium. We then prove 
that for every $\ep > 0$ there is an $\ep$-solvable subgame.

\begin{definition}
Let $\ep >0$ 
and let $h\in H$ be a history.
The subgame $\Gamma_h$ is \emph{$\ep$-solvable} if there is an $\ep$-equilibrium in $\Gamma_h$.
\end{definition}

\begin{theorem}
\label{theorem:easy}
In every stochastic game, for every $\ep>0$, there is an $\ep$-solvable subgame.
\end{theorem}

Note that the $\ep$-solvable subgame that is guaranteed to exist by Theorem~\ref{theorem:easy} may depend on $\ep$.

Our result connects and gives a very partial answer to the long-standing open problem whether every multiplayer stochastic game 
with finite action sets, finite or countably infinite
state space,  
and bounded and Borel-measurable payoffs admits an $\ep$-equilibrium, for every $\ep>0$. As mentioned in the Introduction, a subgame-perfect $\ep$-equilibrium does not always exist, so there are stochastic games in which there is no single strategy profile that induces an $\ep$-equilibrium in all subgames simultaneously.

When the payoff functions of the players satisfy the condition of \emph{shift-invariance} (also called \emph{prefix-independence}),
we can strengthen Theorem~\ref{theorem:easy}. 
The payoff function $f_i$ for player $i\in I$ is called shift-invariant, if $f_i(s^1,a^1,s^2,a^2,s^3,a^3,\ldots)=f_i(s^2,a^2,s^3,a^3,\ldots)$ for every run $(s^1,a^1,s^2,a^2,s^3,a^3,\ldots)\in\mathscr{R}$. Equivalently, $f_i$ is shift-invariant if whenever two runs have the form $hr$ and $h'r$, i.e., they only differ in the prefixes $h$ and $h'$, then $f_i(hr) = f_i(h'r)$. The set of shift-invariant
functions is not included by, neither does it include, the set of Borel-measurable functions, see Rosenthal (1975) and Blackwell and Diaconis (1996).
Many evaluation functions in the literature of dynamic games are shift-invariant, such
as the long-run average payoff (cf.~Remark~\ref{undisc}), and the limsup of
stage payoffs (e.g., Maitra and Sudderth (1993)). Various classical winning conditions in the
computer science literature, such as the B\"{u}chi, co-B\"{u}chi, parity, Streett, and M\"{u}ller (see, e.g.,
Horn and Gimbert (2008), Chatterjee and Henzinger (2012), or Bruy\'{e}re (2021)) are also shift-invariant.
The discounted payoff (e.g., Shapley (1953)) is not shift-invariant.

When the payoff functions of the players are all shift-invariant, and a subgame at some history $h=(s^1,a^1,s^2,a^2,\ldots,s^n)$ is $\ep$-solvable, then by shift-invariance, there is an $\ep$-equilibrium for the initial state $s^n$. This implies the following corollary of Theorem \ref{theorem:easy}.

\begin{corollary}
\label{corollary:easy}
Suppose that player $i$'s payoff function $f_i$ is 
bounded, Borel-measurable, and
shift-invariant, for each $i \in I$.
Then, 
for every $\ep>0$, there is an initial state $s \in S$ for which an $\ep$-equilibrium exists.
\end{corollary}

Since the long-run average payoff is shift-invariant, Corollary~\ref{corollary:easy} implies the results of Thuijsman and Vrieze (1991) and Vieille (2000c) as mentioned above.\bigskip

\begin{proof}[Proof of Theorem~\ref{theorem:easy}]
Assume w.l.o.g.~that all payoffs are between 0 and 1, that is, $f_i(r)\in[0,1]$ for each player $i\in I$ and each run $r\in\mathscr{R}$. 

Fix $\ep \in (0,1]$
and let $\delta \in (0,1)$ be sufficiently small so that
\begin{equation}\label{choice-delta}
5\delta+4(|I|+1)\delta^{\frac{1}{4}} \,<\, \ep.
\end{equation}
For each player $i \in I$, let $D_i^\delta$ be the function given by Theorem~\ref{theorem:martin} (for the version with the minmax value). Let $D^\delta=(D^\delta_i)_{i\in I}$. 
Let $\sigma^*$ be the minmax $\delta$-acceptable strategy profile
given in Theorem~\ref{theorem:acceptable}:
for each history $h \in H$, the mixed action profile
$\sigma^*(h)$ is an equilibrium in the one-shot game $G_O(D^\delta,h)$.

For each history $h \in H$, denote by $C(h) \subseteq \calR$ the cylinder set defined by $h$ in the Borel sigma-algebra $\calB(\calR)$:
\[ C(h) := \{ r \in \calR \colon h \prec r\}. \]
For each $n \in \dN$, denote by $\calF^n$ 
the sigma-algebra over $\calR$ that is defined by histories in stage $n$;
that is, $\calF^n$ is the minimal sigma-algebra that contains,
for each history $h$ in stage $n$, the set $C(h)$.

Throughout the proof we will define various functions $\varphi$ whose domain is $\calR$
and that satisfy the following condition:
there is a nonempty set $Z \subseteq H$ of histories such that for every $h \in Z$, the function $\varphi$ is constant over $C(h)$.
In such a case, we will denote by $\varphi(h)$ the value of $\varphi$ on $C(h)$.
Two cases that satisfy this condition are:
\begin{enumerate}
    \item When $(Y^n)_{n \in \dN}$ is a stochastic process defined on $\calR$ and adapted to the filtration $(\calF^n)_{n \in \dN}$, for each $n,k\in\dN$ with $k\geq n$, the function $Y^n$ satisfies this condition with respect to the set $Z$ that consists of all histories in stage $k$.
    \item When $\theta : \calR \to \dN \cup \{\infty\}$ is a stopping time adapted to the filtration $(\calF^n)_{n \in \dN}$, the function $\theta$ satisfies this condition with respect to the set $Z$ that consists of all histories $r^{\theta(r)}$, i.e., the prefix of $r$ up to stage $\theta(r)$, for all $r\in\mathscr{R}$ satisfying $\theta(r) < \infty$.
\end{enumerate}

\bigskip
\noindent\textbf{General idea:}
We will use the strategy profile $\sigma^*$ to 
identify an $\ep$-solvable subgame and 
derive an $\ep$-equilibrium
in this subgame.
The strategy profile $\sigma^*$ itself is not an $\ep$-equilibrium,
because
there are some ways in which a player, say player~$i$, may be able to profit by deviating from $\sigma^*$:
\begin{itemize}
\item 
 Indeed, player~$i$'s payoff may not be constant on the support of $\sigma^*$, and 
hence, among the actions that receive a positive probability under $\sigma_i^*$, player~$i$ may prefer some actions to others. We will deal with this problem as follows. 
The boundedness and Borel-measurability of $(f_i)_{i \in I}$ imply 
that there is a history $h \in H$ such that
on $C(h)$, with $\prob_{h,\sigma^*}$-probability close to 1,
the payoffs of all players are almost constant.
Since all probability measures over $\calR$ are regular,
there is a compact set $K\subseteq C(h)$ such that $K$ has $\prob_{h,\sigma^*}$-probability close to 1,
and thus on $K$ the payoff functions of all players are almost constant.
We ensure that no player deviates to a play outside $K$,
by instructing the players to punish a deviator if a play outside $K$ is reached.
Since $K$ is closed, the fact that the realized play is outside $K$ is known in finite time.
To identify the deviator, we will use a recent result by Alon, Gunby, He, Shmaya, and Solan (2022).

\item
 Since punishment strategies are used against deviations, it is essential that the punishment is effective. This is only the case if player~$i$'s minmax value at the history where the punishment starts is not much higher than her expected payoff upon following $\sigma^*$.
To profit, player~$i$ may use this observation and deviate in such a way that leads to a history where 
her payoff is high 
and her minmax value is high as well,
so that player $i$ improves her payoff even when punished.
Theorem~\ref{theorem:martin} implies that the minmax value is almost a submartingale under $\sigma^*$,
hence there is a history $h\in H$ such that in the subgame $\Gamma_h$, with $\prob_{h,\sigma^*}$-probability close to 1 the minmax value of all players is almost constant.
We will add a test that verifies that no player~$i$ plays in a way that increases the probability to reach a history where her minmax value is high.
A player who fails this test will be punished in an effective way.
\end{itemize}

We now turn these ideas into a formal proof.

\noindent\textbf{Step 1:} Representing the probability to reach a given set of histories.

In this step, we prove a certain formula for the probability to reach a given set of histories.
Let $Q \subseteq H$ be a set of histories such that no history in $Q$ is an extension of another history in $Q$: there are no $h,h'\in Q$ with $h\prec h'$.
Let $\sigma$ be a strategy profile,
let $n\in\dN$, and let $\theta : \calR \to \dN \cup \{\infty\}$ be a stopping time. 
Recall that the prefix of a run $r\in\mathscr{R}$ in stage $k$ is denoted by $r^k$. 
We will provide a formula for 
$\prob_\sigma(r^k \in Q \hbox{ for some } k \in\{n,n+1,\ldots,\theta(r)\})$,
which is the probability under $\sigma$ that the run reaches a history in $Q$ in one of the stages between $n$ and the stopping time.

Fix a player $i \in I$. For every history $h\in H$ and every action $a_i \in A_i$, define
\begin{equation}
\label{equ:lambda}
\Lambda_i(h,a_i) \ := \ 
\sum_{\{a_{-i} \in A_{-i},\, s \in S \colon (h,(a_i,a_{-i}),s) \in Q\}}\ \sigma_{-i}(a_{-i}\mid h) \cdot p(s \mid s_h,a_i,a_{-i}),
\end{equation} 
which is the probability that the history in the next stage is in $Q$ (and then this is the first stage when the history is in $Q$), when player~$i$ selects the action $a_i$ and the other players follow $\sigma_{-i}$. Recall that $\sigma_{-i}(a_{-i}\mid h)$ denotes the probability assigned to the action profile $a_{-i}$ under the strategy profile $\sigma_{-i}$ at the history $h$, and $p(s \mid s_h,a_i,a_{-i})$ is the probability that the next state is $s$ when, in the current state $s_h$ at the history $h$, the action profile is $(a_i,a_{-i})$.
Note that $\Lambda_i(h,a_i) = 0$ whenever a prefix of $h$ lies in $Q$.
To save cumbersome notations, we do not specifically mentioned the dependency of $\Lambda_i(h,a_i)$ on $Q$ and $\sigma_{-i}$.

For every run $r = (a^1,a^2,\ldots) \in \calR$, every $n\in\dN$ with $n\geq 2$, and every $\ell \in \dN \cup \{\infty\}$, define
\begin{equation}
\label{equ:easy:42}
\zeta_i(r,n,\ell) \,:=\, \sum_{k = n}^{\ell}\Lambda_i(r^{k-1},a^k_i) 
\in \dR_+
\end{equation}
if $n\leq \ell$, and define $\zeta_i(r,n,\ell)=0$ if $n>\ell$.
Though the quantity $\zeta_i(r,n,\ell)$ 
may be larger than 1,
it can be thought of as
a fictitious probability that the run could have reached a history in $Q$ at any of the stages $n,\ldots,\ell$,
given the actual run $r$ and assuming that player $i$'s opponents follow $\sigma_{-i}$. 
A quantity similar to $\zeta_i$ was defined in Flesch and Solan (2022a)
in their study of two-player stochastic games.

As we now show,
the \emph{expectation} of $\zeta_i$ is indeed the probability to reach $Q$.
Specifically, we argue that for every $n\in\dN$ and every stopping time $\theta:\mathscr{R}\to\dN\cup\{\infty\}$,
\begin{equation}
\label{equ:easy:1}
\E_{\sigma}\bigl[\zeta_i(r,n,\theta(r)) \bigr] \,=\, \prob_{\sigma}\bigl(r^k \in Q \hbox{ for some }k\in\{n,n+1,\ldots,\theta(r)\} \bigr). 
\end{equation}
Indeed, 
for every history $h$, let
$\theta(\neg h)$ be a Boolean variable that is true if and only if
$\theta$ does not stop along the history $h$ (that is, $\theta(r)>\len(h)$ for each $r\succ h$). 
Then,
\begin{align}
\label{equ:easy:71}
\E_{\sigma}&\bigl[\zeta_i(r,n,\theta(r)) \bigr] 
\,=\,
\E_{\sigma}\left[ \sum_{k=n}^{\theta(r)} \Lambda_i(r^{k-1},a^k_i) \right]\\[0.1cm]
\label{equ:easy:72}
&=\,
\sum_{k=n}^\infty \ 
\sum_{\{h \in H\colon \len(h) = k-1,\ \theta(\neg h)\}}\ \prob_{\sigma}\bigl(h\bigr) \cdot\left[
\sum_{a_i \in A_i} \sigma_i(a_i\mid h) \cdot \Lambda_i(h,a_i)\right]\\[0.1cm]
\label{equ:easy:73}
&=\,
\prob_{\sigma}\bigl(r^k \in Q \hbox{ for some }k\in\{n,n+1,\ldots,\theta(r)\} \bigr),
\end{align}
where \Eqref{equ:easy:71} follows from the definition of $\zeta_i$,
\Eqref{equ:easy:72} follows from changing the order of summation,
and \Eqref{equ:easy:73} holds by the definition of $\Lambda_i$.

\bigskip

\noindent\textbf{Step 2:} Identifying a history $h^*$, or equivalently, a subgame $\Gamma_{h^*}$, and identifying a target equilibrium payoff $c$.

The process $(\E_{\sigma^*}\bigl[f_i \mid \calF^n])_{n \in \dN}$ is a martingale that converges $\prob_{\sigma^*}$-a.s.~to $f_i$,
for each player $i\in I$, and hence denoting $f=(f_i)_{i\in I}$,
there is $n_1 \in \dN$ such that%
\footnote{The constant $\frac{2}{3}$ in \Eqref{equ:easy:16} 
and Eqs.~\eqref{equ:easy:17} and~\eqref{equ:easy:18} can be replaced by any three constants whose sum is larger than 2.}
\begin{equation}
\label{equ:easy:16}
\prob_{\sigma^*}\Bigl( \bigl\|\,f(r) - \E_{r^n,\sigma^*}[f ]\,\bigr\|_\infty \leq \delta, \ \ \ \forall n \geq n_1 \Bigr) \,>\, \frac{2}{3}; 
\end{equation} 
thus, under the strategy profile $\sigma^*$, with probability more than $\frac{2}{3}$, it holds that in all stages $n\geq n_1$, the expected payoff in the subgame at the history in stage $n$ is close to the realized payoff.

For each player $i \in I$ and stage $n \in \dN$, define a random variable $Y^n_i : \calR \to \mathbb{R}$ by
\[ Y^n_i(r)\, :=\, D^\delta_i(r^n),\qquad \forall r\in\mathscr{R}. \]
Because $\sigma^*(h)$ is an equilibrium in the one-shot game $G_O(D^\delta,h)$ for each history $h\in H$, property~(2) of Theorem~\ref{theorem:martin} implies that 
the process $(Y^n_i)_{n \in \dN}$ is a submartingale under $\prob_{\sigma^*}$: 
\begin{equation}\label{Ysubm}
\E_{\sigma^*}[Y^{n+1}_i\mid{\cal{F}}^n]\,\geq\,Y^n_i,\qquad \forall n\in\dN.
\end{equation}
Hence, $(Y^n_i)_{n \in \dN}$ converges $\prob_{\sigma^*}$-a.s.~to a limit $Y_i^\infty$.
Denote
\[ Y^n(r) \,:=\, (Y^n_i(r))_{i \in I},\qquad \forall r\in\mathscr{R},\, \forall n \in \dN. \]
Since the sequence $(Y^n(r))_{n \in \dN}$ converges $\prob_{\sigma^*}$-a.s.,
there is $n_2 \in \dN$ such that 
\begin{equation}
\label{equ:easy:17}
\prob_{\sigma^*}\bigl( \|Y^k(r) - Y^{n}(r)\|_\infty \leq \delta, \ \ \forall k \geq n \geq n_2 \bigr) \,>\, \frac{2}{3}. 
\end{equation} 

For each player $i \in I$ and stage $n \in \dN$, denote 
\[ W^n_i \,:=\, \E_{\sigma^*}[Y^{n+1}_i \mid \calF^n], \qquad \forall n \in \dN. \]
Using \Eqref{Ysubm}, we have $W^n_i\geq Y^n_i$ for every $n \in \dN$.
Hence, the sequence $(W^n_i)_{n \in \dN}$ is a submartingale under $\prob_{\sigma^*}$. Indeed,
\[\E_{\sigma^*}[W^{n+1}_i \mid \calF^n]
\ \geq\  \E_{\sigma^*}[Y_i^{n+1} \mid \calF^n] \ =\ W_i^n, \qquad\forall n\in\mathbb{N}.
\]
Thus, the sequence $(W^n_i)_{n \in \dN}$ converges $\prob_{\sigma^*}$-a.s.~to $Y_i^\infty$.
Denote
\[ W^n(r) \,:=\, (W^n_i(r))_{i \in I},\qquad \forall r\in\mathscr{R},\, \forall n \in \dN. \]
Since the sequence $(W^n(r))_{n \in \dN}$ converges $\prob_{\sigma^*}$-a.s.,
there is $n_3 \in \dN$ such that 
\begin{equation}
\label{equ:easy:18}
\prob_{\sigma^*}\bigl( \|W^k(r) - W^{n}(r)\|_\infty \leq \delta, \ \ \forall k \geq n \geq n_3 \bigr) \,>\, \frac{2}{3}. 
\end{equation} 

Set $n_0 := \max\{n_1,n_2,n_3\}$, and define
\begin{equation} 
\label{equ:easy:41}
\widehat{\mathscr{R}}_{n_0} := 
\left\{ r \in \calR \colon 
\begin{array}{l}
\|f(r) - \E_{r^n,\sigma^*}[f ]\|_\infty \leq \delta, \ \ \ \forall n \geq n_0\\
\|Y^k(r) - Y^n(r)\|_\infty \leq \delta, \ \ \ \forall k \geq n\geq n_0\\
\|W^k(r) - W^n(r)\|_\infty \leq \delta, \ \ \ \forall k \geq n\geq n_0
\end{array} \right\}.
\end{equation}
By Eqs.~\eqref{equ:easy:16}, \eqref{equ:easy:17},
and~\eqref{equ:easy:18},
\[ \prob_{\sigma^*}\bigl(\widehat{\mathscr{R}}_{n_0}\bigr)\, >\, 0. \]
As a consequence of L\'{e}vy's zero-one law,
there is $n^* \geq n_0$ and a history $h^*\in H$ in stage $n^*$ such that 
\begin{equation}
\label{equ:easy:19}
\prob_{h^*,\sigma^*}\bigl(\widehat{\mathscr{R}}_{n_0}
\bigr) \,>\, 1-\delta.
\end{equation}
Denote
\begin{equation}
\label{equ:easy:c}
c \,:=\, \E_{h^*,\sigma^*}[f] \in \dR^{|I|}. 
\end{equation}
By the definitions of $\widehat{\mathscr{R}}_{n_0}$ and $c$, we have
\begin{equation}
\label{onRn0}
c_i-\delta\,\leq\,f_i(r)\,\leq\,c_i+\delta,\qquad \forall i\in I,\,\forall r\in\widehat{\mathscr{R}}_{n_0}\cap C(h^*),
\end{equation}
and by \Eqref{lowerpayoff}, we have
\begin{equation}
\label{equ:easy:20}
c_i \,=\, 
\E_{h^*,\sigma^*}[ f_i]
\,\geq\,\overline v_{O,i}(D^\delta_i,h^*)\,\geq\,
D^\delta_i(h^*)\,\geq\,\overline{v}_i(h^*)-\delta,\qquad\forall i \in I.
\end{equation}

We will construct an $\ep$-equilibrium in the subgame $\Gamma_{h^*}$ with payoff close to $c$.

\bigskip
\noindent
\textbf{Step 3:} Histories with high minmax value.

For each player $i \in I$ denote by $Q_i$ the set of histories $h\in H$ such that (i) $h\succeq h^*$, (ii) $\overline v_i(h) > c_i + 3\delta$, and (iii) $\overline{v}_i(h')\leq c_i+3\delta$ for each $h'\in H$ with $h^*\preceq h'\prec h$.
We interpret $Q_i$ as the set of histories in the subgame $\Gamma_{h^*}$ where player $i$'s minmax value is high for the first time.

For each player $i\in I$ and run $r \in \calR$,
let $m_i(r)$ be the entry time to $Q_i$: If $r$ has a prefix that belongs to $Q_i$, then this prefix is unique, and $m_i(r)$ is thus the unique stage such that $r^{m_i(r)}\in Q_i$. If $r$ has no prefix in $Q_i$ then $m_i(r)=\infty$. 

For each $i \in I$,
define the quantities $\Lambda_i(h,a_i)$ and $\zeta_i(r,n,\ell)$
as in Eqs.~\eqref{equ:lambda} and~\eqref{equ:easy:42} with respect to the set $Q_i$. We will be interested in the case when $n=n^*$, i.e., the stage of history $h^*$.

By taking $\theta = \infty$ in \Eqref{equ:easy:1},
we obtain for each player $i \in I$ and each strategy $\sigma_i$,
\begin{eqnarray}
\nonumber
\E_{h^*,\sigma_i,\sigma^*_{-i}}\bigl[\zeta_i(r,n^*,\infty)\bigr] &\,=\,& \prob_{h^*,\sigma_i,\sigma^*_{-i}}\bigl(r^k \in Q_i \hbox{ for some } k\in\{n^*,n^*+1,\ldots\}\bigr)\\[0.1cm]
&=& \prob_{h^*,\sigma_i,\sigma^*_{-i}}\bigl( m_i(r) < \infty \bigr). 
\label{equ:easy:4}
\end{eqnarray}

For each player $i \in I$, let $\nu_i : \calR \to \dN \cup \{\infty\}$ be the stopping time that indicates the first stage after stage $n^*$ at which $\zeta_i$ exceeds the threshold $\sqrt\delta$:
\[ \nu_{i}(r) \,:=\, \min\left\{ k \geq n^* \colon \zeta_i(r,n^*,k) \geq \sqrt\delta \right\}, \]
where the minimum of the empty set is $\infty$.
The intuition is that, when $\nu_i(r)<\infty$, the probability at stage $\nu_i(r)$ that the play could have reached a history with a high minmax value for player~$i$ is higher than $\sqrt\delta$,
and so the other players may suspect that player~$i$ 
is deviating and trying to reach a state where she cannot be punished effectively.

\bigskip
\noindent
\textbf{Step 4:} Identifying a good set of plays $\widehat{\mathscr{R}}\subseteq\mathscr{R}$.

Fix for a moment a player $i \in I$.
A useful property of the set $\widehat{\mathscr{R}}_{n_0}$ is the following:
\begin{equation}
\label{equ:easy:25}
\overline v_i(r^k) \,\leq\, Y_i^{n^*}(r) + 2\delta \,\leq\, c_i + 2\delta, 
\qquad\forall r \in \widehat{\mathscr{R}}_{n_0}\cap C(h^*), \forall k \geq n^*,
\end{equation}
which means that the minmax value of player $i$ in the subgame $\Gamma_{r^k}$ is at most $c_i+2\delta$. 
To see that \Eqref{equ:easy:25} holds, note that for every $r \in \widehat{\mathscr{R}}_{n_0}\cap C(h^*)$ and every $k \geq n^*$,
\begin{eqnarray}
\label{equ:easy:26}
\overline v_i(r^k)
&\leq& D^\delta_i(r^k) + \delta\\ 
\label{equ:easy:26a}
&=&
Y^k_i(r) + \delta\\
\label{equ:easy:27}
&\leq& Y_i^{n^*}(r) + 2\delta\\
\label{equ:easy:28}
&=& D^\delta_i(h^*) + 2\delta\\
\label{equ:easy:29}
&\leq& c_i + 2\delta,
\end{eqnarray}
where \Eqref{equ:easy:26} holds by property~(1) of Theorem~\ref{theorem:martin},
\Eqref{equ:easy:26a} holds by the definition of $Y_i^k$,
\Eqref{equ:easy:27} holds because $r \in \widehat{\mathscr{R}}_{n_0}$ and $k\geq n^*\geq n_0$,
\Eqref{equ:easy:28} holds since $r\in C(h^*)$ and by the definition of $Y_i^{n^*}$,
and \Eqref{equ:easy:29} holds by \Eqref{equ:easy:20}.

We also have 
\begin{equation} 
\label{equ:77}
\prob_{h^*,\sigma^*}\bigl(m_i(r) \,<\, \infty \bigr) 
\,\leq\, \prob_{h^*,\sigma^*}\bigl(\mathscr{R}\setminus \widehat{\mathscr{R}}_{n_0} \bigr) 
\,<\, \delta. 
\end{equation}
Indeed, the second inequality follows from \Eqref{equ:easy:19}. 
To see that the first inequality holds, we show that if $m_i(r)<\infty$ for some run $r\in C(h^*)$, then 
necessarily
$r\in \mathscr{R}\setminus \widehat{\mathscr{R}}_{n_0}$. 
Assume then that $m_i(r)<\infty$ and let $k=m_i(r)$. 
It follows that $r^k\in Q_i$, and hence $\overline{v}_i(r^k)>c_i+3\delta$. Thus, by  
\Eqref{equ:easy:25}, we obtain $r\in \mathscr{R}\setminus \widehat{\mathscr{R}}_{n_0}$ as desired.

Substituting $\sigma_i = \sigma^*_i$ in \Eqref{equ:easy:4} and
using \Eqref{equ:77}, we obtain
\begin{equation}
\label{equ:easy:5}
\E_{h^*,\sigma^*}\bigl[\zeta_i(r,n^*,\infty) \bigr]\,<\, \delta. 
\end{equation}
The random variable $\zeta_i(\cdot,n^*,\infty)$ is non-negative, 
and hence, by Markov's inequality and \Eqref{equ:easy:5},
\begin{equation}
\label{equ:easy:51}
\prob_{h^*,\sigma^*}\left( \zeta_i(r,n^*,\infty) \geq \sqrt\delta \right) \,
\leq\, \frac{\E_{h^*,\sigma^*}\bigl[\zeta_i(r,n^*,\infty)\bigr]}{\sqrt\delta} \,<\, \sqrt\delta.
\end{equation}

Define
\begin{equation}
\label{defwidehatR}
\widehat{\mathscr{R}}\, :=\, \widehat{\mathscr{R}}_{n_0} \cap \left\{ r \in C(h^*) \colon \zeta_i(r,n^*,\infty) < \sqrt\delta, \ \ \forall i \in I\right\}.
\end{equation} 
By Eqs.~\eqref{equ:easy:19} and~\eqref{equ:easy:51}, and since $\delta<1$, 
\[ \prob_{h^*,\sigma^*}\bigl(\widehat{\mathscr{R}} \bigr) \,>\, 1-\delta-|I|\sqrt\delta\,>\, 1-(|I|+1)\sqrt\delta\,>\,0. \]

Since the probability measure $\prob_{\sigma^*}$ is regular,
there is a compact set $K \subseteq \widehat{\mathscr{R}}$ such that 
\begin{equation}
\label{equ:easy:7}
\prob_{h^*,\sigma^*}\big(K\bigr) \,>\, 1-(|I|+1)\sqrt\delta \,>\,0.
\end{equation} 
Note that, as $K\subseteq \widehat{\mathscr{R}}_{n_0}\cap C(h^*)$, \Eqref{onRn0} implies
\begin{equation}
\label{1stpropci}
c_i-\delta\,\leq\,f_i(r)\,\leq\,c_i+
\delta\qquad \forall i\in I,\,\forall r\in K.
\end{equation}

The limits $\lim_{n \to \infty} Y^n$ and 
$\lim_{n \to \infty} W^n$ exist and coincide $\prob_{h^*,\sigma^*}$-a.s. By \Eqref{equ:easy:19} we have $\prob_{h^*,\sigma^*}(\widehat{\mathscr{R}}_{n_0})>0$, and hence there is a run $r\in \widehat{\mathscr{R}}_{n_0}\cap C(h^*)$ for which $\lim_{n \to \infty} Y^n(r)$ and 
$\lim_{n \to \infty} W^n(r)$ exist and coincide. 
Therefore, by the definition of $\widehat{\mathscr{R}}_{n_0}$, we have 
$|Y^{n^*}(r) - Y^\infty(r)| \leq \delta$
and
$|W^{n^*}(r) - Y^\infty(r)| \leq \delta$.
It follows that
\begin{equation}
\label{equ:easy:52}
\|Y^{n^*}(h^*) - W^{n^*}(h^*)\|_\infty \,\leq\,
2\delta.
\end{equation}

Since the set $K$ is closed, 
its complement in the subgame starting at the history $h^*$ $K^c:=C(h^*)\setminus K$ is open.
Hence, $K^c = \bigcup_{h \in Z} C(h)$
for some set $Z \subseteq H$ such that each history in $Z$ extends $h^*$.
We assume w.l.o.g.~that $Z$ is minimal in the following sense: (i) there are no histories $h,h'\in Z$ such that $h\prec h'$ (in this case, we can drop $h'$ from $Z$), and (ii) there is no history $h\in H$ such that $(h,a)\in Z$ for each action profile $a\in A$ (in this case, in $Z$ we can replace all $(h,a)$, for $a\in A$, by the single history $h$).

By Alon, Gunby, He, Shmaya, and Solan (2022),
there is a function $g : K^c \to I$,
where $g(r)$ depends only on the prefix of $r$ that lies in $Z$,
such that
\begin{equation}
\label{equ:71}
\prob_{h^*,\sigma_i,\sigma^*_{-i}}\bigl( r\in K^c \hbox{ and } g(r) \neq i \bigr) \,\leq\, 2|I|\delta^{1/4}\,=:\, \eta,
\qquad\forall i \in I,\, \forall \sigma_i \in \Sigma_i. 
\end{equation} 
The function $g$ is called a \emph{blame function}, and its interpretation is as follows: For every run $r\in K^c$, the function $g$ selects a player, who is blamed for the fact the the run $r$ left the set $K$. In view of \Eqref{equ:71}, the probability that $g$ blames an innocent player, i.e., a player who truthfully followed her strategy in $\sigma^*$, is at most $\eta$. The existence of such a blame function is a key-step in our proof, as it tells us which player to punish when the run leaves $K$.

\begin{comment} \rm
Alon, Gunby, He, Shmaya, and Solan (2022) state their result for repeated games with a finite number of players and finite action sets.
We here explain how their result can be applied in our setup.

To accommodate the state variable, we add a new fictitious player to the game, and divide each stage into two substages.
In the first substage, the players select actions.
In the second substage, the new fictitious player 
selects the state according to the probability distribution dictated by the last state she selected and the actions of the ``real'' players.
In our setup where the set of states may be countably infinite,
this would require this fictitious player to have a countable action set.
To overcome this problem, so that we can  apply Alon, Gunby, He, Shmaya, and Solan (2022),
we define an auxiliary repeated game where the set of actions of this fictitious player is finite.

 Fix $\rho > 0$. For each state $s \in S$, each action profile $a \in A$, and each stage $n \in \dN$, select a finite set of states that $\widehat S(s,a,n)$ that satisfies $p(\widehat S(s,a,n)\mid s,a) \geq 1-\rho/2^n$.
 Define an auxiliary repeated game where the action set of the fictitious player at the second substage of stage $n$ is $\widehat S(s,a,n)$,
where $s$ is the action the fictitious player selected at stage $n-1$ and $a$ is the action profile selected by the players at the first substage of stage $n$; at stage $n=1$ the state $s$ is the initial state.
Since the total probability that in the original stochastic game, at any stage $n\in\mathbb{N}$, a state not in $\widehat S(s,a,n)$ will ever be selected is at most $\rho$,
this procedure adds $\rho$ to our approximations.

Finally, one has to note that the proof of Alon, Gunby, He, Shmaya, and Solan (2022) is valid also in the setup where the action sets of the players may be history dependent (yet finite).

Since $\rho$ is arbitrary, one can in fact get rid of the additional term of  $\rho$, and derive the bound given in Eq.~\eqref{equ:71}. \hfill$\Diamond$
\end{comment}

Define the stopping time $\theta^K : \calR \to \dN \cup \{\infty\}$ as the stage in which the play leaves $K$:
\[ \theta^K(r) \,:=\, \min\bigl\{ k \in \dN \colon r^k \in Z \bigr\}, \]
where the minimum of the empty set is $\infty$. Note that \[K^c\,=\,\{ r \in C(h^*) \colon \theta^K(r) < \infty\}.\]
For a history $h=(s^1,a^1,\ldots,s^{n-1},a^{n-1},s^n)$ extending $h^*$, we similarly define
\[ \theta^K(h) \,:=\, \min\bigl\{ k \in \{1,\ldots,n\} \colon (s^1,a^1,\ldots,s^{k-1},a^{k-1},s^k) \in Z \bigr\}. \]

\bigskip
\noindent\textbf{Step 5:} Definition of a strategy profile $\widehat \sigma$.

We are now ready to define a strategy profile $\widehat \sigma$ in the subgame $\Gamma_{h^*}$.
We will then prove that this strategy profile is an $\ep$-equilibrium in this subgame.

The strategy profile $\widehat\sigma$ coincides with $\sigma^*$, with one modification that ensures that no player can profit too much by deviating. Suppose that the current history is $h\succeq h^*$.
\begin{itemize}
\item
If no prefix of $h$, including $h$, belongs to $Z$, or equivalently, if $\theta^K(h)=\infty$, then $\widehat\sigma_i(h)=\sigma^*_i(h)$ for each player $i\in I$.
\item
If the history $h$ reaches $Z$, or equivalently, if $\theta^K(h)=\len(h)$, then by using the blame function $g$, 
declare player~$g(h)$ the deviator.
From this stage and on,
the opponents of player $g(h)$ punish player~$g(h)$ at her minmax level $\overline v_{g(h)}(h)$ plus $\delta$,
that is,
the opponents switch to a strategy profile $\sigma'_{-g(h)}$
that satisfies
\[ \E_{h,\sigma_{g(h)},\sigma'_{-{g(h)}}}\bigl[ f_{g(h)} \bigr] \,\leq\, \overline v_{g(h)}(h) + \delta, \qquad\forall \sigma_{g(h)}\in\Sigma_{g(h)}. \]
\end{itemize}

\bigskip
\noindent\textbf{Step 6:} $\widehat \sigma$ is an $\ep$-equilibrium in $\Gamma_{h^*}$.

We start by calculating the expected payoff under $\widehat\sigma$.
By Eqs.~\eqref{1stpropci} and \eqref{equ:easy:7} and since all payoffs are between 0 and 1,
\begin{eqnarray}
\nonumber
\E_{h^*,\widehat\sigma}\bigl[f_i \bigr] 
&\geq& \prob_{h^*,\sigma^*}(K)\cdot \E_{h^*,\sigma^*}\bigl[f_i \mid K\bigr]+(1-\prob_{h^*,\sigma^*}(K))\cdot 0\\ 
\nonumber &\geq& (1-(|I|+1)\sqrt\delta) \cdot (c_i-\delta) \\
&\geq& c_i - (|I|+1)\sqrt\delta-\delta\\
&\geq& c_i - 2(|I|+1)\sqrt\delta.\label{equ:easy:39}
\end{eqnarray}

We next show that no player can profit more than $\ep$ by deviating.
Fix then a player $i \in I$ and a strategy $\sigma_i$.
To calculate $\E_{h^*,\sigma_i,\widehat\sigma_{-i}}\bigl[ f_i \bigr]$,
we will divide the set $C(h^*)$ into four subsets,
and bound player~$i$'s payoff from above on each of these sets.

\begin{itemize}
\item \textbf{First subset:} The set $E_1=K$.

On the set $E_1$, no player is punished. By \Eqref{1stpropci}, 
\begin{equation}
\label{equ:easy:10}
\E_{h^*,\sigma_i,\widehat\sigma_{-i}}\bigl[ f_i\cdot \chi_{E_1}\bigr] \,\leq\, \prob_{h^*,\sigma_i,\widehat\sigma_{-i}}\bigl( E_1\bigr)\cdot (c_i + \delta),
\end{equation}
where $\chi_{W}$ denotes the characteristic function of the set $W\subseteq\mathscr{R}$.
\end{itemize}

The following subsets will deal with the set $K^c=C(h^*)\setminus K$, where the player given by the blame function $g$ is punished at her minmax level. 
\begin{itemize}
\item \textbf{Second subset:} The set $E_2 := K^c \cap \{g(r) \neq i\}$.

On $E_2$, a player who is not player~$i$ is punished,
and the only upper bound we have on player~$i$'s payoff is the general upper bound, which is 1.
Fortunately, the event $E_2$ occurs with small probability.
Indeed, by \Eqref{equ:71},
\begin{equation}
\label{equ:easy:11}
\E_{h^*,\sigma_i,\widehat\sigma_{-i}}\bigl[ f_i\cdot \chi_{E_2}\bigr]\,\leq\,\prob_{h^*,\sigma_i,\widehat\sigma_{-i}}\bigl( E_2 \bigr)\,=\,\prob_{h^*,\sigma_i,\sigma^*_{-i}}\bigl( E_2 \bigr)\, \leq\, \eta. 
\end{equation}

\item
\textbf{Third subset:} The set 
$E_3 := K^c \cap \{g(r) = i\} \cap \{\nu_i(r) = \theta^K(r)\}$.

On the set $E_3$, player~$i$ is punished as 
the quantity $\zeta_i$ becomes large: $\zeta_i(r,n^*,\theta^K(r))$ is at least $\sqrt\delta$. This means intuitively that the actions of player $i$ have made it likely that the run leaves the set $K$. Note that we do not rule out the case that $\nu_j(r) = \theta^K(r)$ for some additional player $j \neq i$.

Let $h \succeq h^*$ be a history in some stage $k\geq n^*$ such that none of its prefixes (including $h$ itself) is in $Z$; that is, $\theta^K(h)=\infty$.
Suppose that player $i$ plays action $a_i$ in stage $k$, and from stage $k+1$ on the players $I \setminus \{i\}$ punish player~$i$ at her minmax level plus $\delta$.
Then, player~$i$'s payoff is at most
\begin{align}
\sum_{a_{-i}  \in A_{-i}} 
\sigma^*_{-i}(a_{-i}\mid h)& \cdot \left( \overline v_i(h,a_{i},a_{-i}) + \delta \right)\nonumber\\
&\leq\,
\sum_{a_{-i}  \in A_{-i}} 
\sigma^*_{-i}(a_{-i}\mid h)\cdot\left( D^\delta_i(h,a_{i},a_{-i}) + 2\delta \right)\label{equ:easy:32}\\
&\leq\,
\sum_{a  \in A}
\sigma^*(a\mid h)\cdot\left( D^\delta_i(h,a) + 2\delta \right)\label{equ:easy:32-2}\\
\label{equ:easy:33}
&=\, W^k_i(h) + 2\delta\\
\label{equ:easy:37}
&\leq\, W^{n^*}_i(h^*) + 3\delta\\
\label{equ:easy:36}
&\leq\, Y^{n^*}_i(h^*) + 5\delta\\
\label{equ:easy:34}
&\leq\, c_i + 5\delta,
\end{align}
where \Eqref{equ:easy:32} holds by property~(1) in Theorem~\ref{theorem:martin},
\Eqref{equ:easy:32-2} holds since $\sigma^*(h)$ is an equilibrium in the one-shot one-shot game $G_O(D^\delta,h)$, \Eqref{equ:easy:33} holds by the definition of $W^k_i$,
\Eqref{equ:easy:37} holds since on $K \subseteq \widehat{\mathscr{R}} \subseteq \widehat{\mathscr{R}}_{n_0}$ we have 
$|W^k_i(r) - W^{n^*}_i(r)| \leq \delta$,
\Eqref{equ:easy:36} holds by \Eqref{equ:easy:52},
and \Eqref{equ:easy:34} holds by \Eqref{equ:easy:25}.

Let now $r \in E_3$ and set $k := \theta^K(r)$.
Denote by $h := r^k$ the prefix of $r$ at which the run leaves $K$,
and write $h = (h^*, a^{n^*},s^{n^*+1},a^{n^*+1},\ldots, s^k)$.
Since $k=\theta^K(r)=\nu_i(r)$, we have $\zeta_i(r,n^*,k) \geq \sqrt\delta$.

Denote by $h' = (h^*, a^{n^*},s^{n^*+1},a^{n^*+1},\ldots, s^{k-1})$
the history that precedes $h$.
Since $k-1<\theta^K(r)$, at $h'$ the run $r$ does not leave $K$.
In particular,
$\zeta_i(h',n^*,k-1) < \sqrt\delta$.
The definition of $\zeta_i$ implies that $\zeta_i(h,n^*,k)$ depends neither on the actions selected by player $i$'s opponents at stage $k-1$ nor on $s_k$, see \Eqref{equ:easy:42}.
That is, for each action profile $a_{-i} \in A_{-i}$ and each state $s \in S$ we have
\[ \zeta_i\bigl( (h',a^k_i,a_{-i},s), n^*,k\bigr)
\,=\, \zeta_i( h, n^*,k) \,\geq\, \sqrt\delta. \]
Therefore, by \Eqref{defwidehatR} and since $K\subseteq \widehat{\mathscr{R}}$, we have $(h',a^k_i,a_{-i},s) \not\in \widehat{\mathscr{R}}$ for each $a_{-i} \in A_{-i}$ and each $s \in S$.
In particular, if at the history $h'$ player~$i$ selects the move $a^k_i$,
then whichever actions the other players select at stage $k-1$ and whichever state the run will reach at stage $k$,
the run will leave $K$ at stage $k$,
and some player (not necessarily player~$i$)
will be punished.
Yet, 
by \Eqref{equ:71},
the probability that a player in $I \setminus \{i\}$ will be punished is at most $\eta$.
This implies, by using \Eqref{equ:easy:34} and by the fact that the payoff is at most 1, that 
\begin{equation}
\label{equ:easy:35}
\E_{h^*,\sigma_i,\widehat\sigma_{-i}}\bigl[ f_i \cdot \chi_{E_3} \bigr] \,\leq\, \prob_{h^*,\sigma_i,\widehat\sigma_{-i}}\bigl(E_3 \bigr)\cdot (c_i + 5\delta) + \eta\cdot 1.
\end{equation}

\item
\textbf{Fourth subset:} The set 
$E_4 := K^c \cap \{g(r) = i\} \cap \{\nu_i(r) > \theta^K(r)\}$.

Roughly, on $E_4$, a deviation is announced upon leaving the set $\widehat\calR_{n_0}$ and player $i$ gets punished. We will show that, with high probability, the punishment is effective.

Note that $E_4$ is the disjoint union of the sets $C(h)$, where $h\in Z$, the blame function $g$ is declaring player $i$ the deviator at the history $h$ (it does not depend on the continuation after $h$), and $\zeta_i$ at $h$ is still below $\sqrt\delta$. Let $Z_4\subseteq Z$ denote the set of these histories. At a history $h\in Z_4$, punishment against player $i$ is effective if player $i$'s minmax value is not high, i.e., $h\notin Q_i$, and thus the set of such histories in $Z_4$ where punishment is not effective is
\[ Q_i^4 \,:=\, Z_4\cap Q_i. \]
We have
\begin{align}
\nonumber\prob_{h^*,\sigma_i,\widehat\sigma_{-i}}\bigl(r^k \in Q_i^4
&\hbox{ for some } k\in\{n^*,n^*+1,\ldots,\theta^K(r)\}\bigr)\\
\nonumber &=\,\prob_{h^*,\sigma_i,\sigma^*_{-i}}\bigl(r^k \in Q_i^4 
\hbox{ for some } k\in\{n^*,n^*+1,\ldots,\theta^K(r)\}\bigr)\\
\label{equ:easy:13}
&\leq\, 
\E_{h^*,\sigma_i,\sigma^*_{-i}}
\bigl[ \zeta_i(r,n^*,\theta^K(r)) \bigr]\\
&\leq\, \sqrt\delta,
\label{equ:easy:12}
\end{align} 
where \Eqref{equ:easy:13} holds by \Eqref{equ:easy:1}
and since the quantity $\zeta_i(r,n^*,\theta^K(r))$ that corresponds to $Q_i^4$ (and $\sigma^*_{-i})$ cannot be larger than the quantity $\zeta_i(r,n^*,\theta^K(r))$ that corresponds to $Q_i$ (because $Q_i^4 \subseteq Q_i)$,
and \Eqref{equ:easy:12} holds since on $E_4$ we have $\nu_i>\theta^K$ and thus $\zeta_i(r,n^*,\theta^K(r)) < \sqrt\delta$.

For each history $h\in Z_4\setminus Q_i$, by the definition of $Q_i$, we have $\overline{v}_i(h)\leq c_i+3\delta$. Hence, by \Eqref{equ:easy:12}, 
\begin{equation}
\label{equ:easy:set4}
\E_{h^*,\sigma_i,\widehat\sigma_{-i}}\bigl[ f_i \cdot \chi_{E_4} \bigr] \,\leq\, \prob_{h^*,\sigma_i,\widehat\sigma_{-i}}\bigl(E_4 \bigr)\cdot (c_i + 3\delta) + \sqrt\delta\cdot 1.
\end{equation}
\end{itemize}

Eqs.~\eqref{equ:easy:10}, \eqref{equ:easy:11}, \eqref{equ:easy:35}, and \eqref{equ:easy:set4} imply that
\[ 
\E_{h^*,\sigma_i,\widehat\sigma_{-i}}
\bigl[ f_i \bigr] \,=\,\sum_{j=1}^4\, \E_{h^*,\sigma_i,\widehat\sigma_{-i}}
\bigl[ f_i \cdot\chi_{E_j}\bigr]
\,\leq\,(c_i+5\delta)+\eta+\sqrt\delta.\]
Together with Eqs.~\eqref{equ:easy:39} and \eqref{choice-delta},
this yields that player $i$'s gain by devaiting is at most 
\[\E_{h^*,\sigma_i,\widehat\sigma_{-i}}[f_i]-\E_{h^*,\widehat\sigma}[f_i]\,\leq\,5\delta+\eta+\sqrt\delta+2(|I|+1)\sqrt\delta\,\leq\,\ep.\]
Hence, $\widehat\sigma$ is an $\ep$-equilibrium in the subgame $\Gamma_{h^*}$, as claimed.
\end{proof}

\section{Discussion}
\label{section:discussion}

The goal of this paper was to present four different existence results for multiplayer stochastic games with general payoff functions: the existence of subgame $\ep$-maxmin strategies, minmax $\ep$-acceptable strategy profiles, extensive-form correlated $\ep$-equilibria, and $\ep$-solvable subgames.

The main tool for each proof was the Martin function, which is an auxiliary function that assigns a real number to each history, and thereby induces a suitable one-shot game at each history of the stochastic game. This function was invented and its existence was proven in Martin (1998), see also Maitra and Sudderth (1998), in the context of two-player zero-sum games, and later it was generalized to multiplayer games in Ashkenazi-Golan, Flesch, Predtetchinski, and Solan (2022a). As discussed in Example~\ref{example:mn},
this function is in fact also an extension of the technique developed by Mertens and Neyman (1981) to prove the existence of the uniform value in two-player zero-sum stochastic games,
and by Neyman (2003) to prove the existence of the uniform minmax value and the uniform maxmin value in multiplayer stochastic games. We refer to the Introduction for recent papers where the Martin function was used for multiplayer repeated games and for multiplayer stochastic games with general payoff functions.

We provided short and straighforward proofs, based on the Martin function, for the existence of subgame $\ep$-maxmin strategies and of extensive-form correlated $\ep$-equilibria; these statements were originally proven by Mashiah-Yaakovi (2015) using different tools. The existence of minmax $\ep$-acceptable strategy profiles and the existence of $\ep$-solvable subgames,
which were proven respectively by Solan (2018) and Vieille (2000c) for the case of the long-run average payoff,
are new for general payoff functions.

The most complicated proof in this paper is the proof for the existence of $\ep$-solvable subgames; cf.~Theorem~\ref{theorem:easy}. This proof uses several ideas
that are not needed for the specific case of 
finitely many states and the 
long-run average payoffs. 
One such idea is a test that verifies that it has only low probability to reach a history where the minmax value of some player is high. 
Such a test was already used in Flesch and Solan (2022a) in the context of two-player stochastic games with shift-invariant payoffs.

A second idea is to approximate a given set of ``good'' runs
by a closed subset.
This approximation allows us to identify a deviation in finite time,
and punish the deviator.
Such an approximation was already used in various papers, such as in Simon (2007), Shmaya (2011), Ashkenazi-Golan, Flesch, Predtetchinski, and Solan (2022a,b), Ashkenazi-Golan, Flesch, and Solan (2022), and Flesch and Solan (2022a,b).

A third idea is the identification of a deviator from the play, when the players use non-pure strategies, based on Alon, Gunby, He, Shmaya, and Solan (2022).
This result was employed in Flesch and Solan (2022b) to provide an alternative proof 
for the existence of $\ep$-equilibria in multiplayer repeated games with tail-measurable payoffs.
As far as we know, Theorem~\ref{theorem:easy} is the first result where the use of this technique is imperative for a proof.

We hope that the Martin function and the existence results presented in this paper will be useful in deriving more results for stochastic games with general payoffs.


\begin{thebibliography}{}
\bibitem{AGHSS}
Alon N., Gunby B., He X., Shmaya E., and Solan E. (2022)
Identifying the deviator.
arXiv:2203.03744.

\bibitem{AKRS1}
Ashkenazi-Golan G., Flesch J., Predtetchinski A., and Solan E. (2022a) 
Existence of equilibria in repeated games with long-run payoffs. 
\emph{Proceedings of the National Academy of Sciences}, \textbf{119}(11), e2105867119.

\bibitem{AKRS2}
Ashkenazi-Golan G., Flesch J., Predtetchinski A., and Solan E. (2022b) 
Regularity of the minmax value and equilibria in multiplayer Blackwell games. 
arXiv:2201.05148.

\bibitem{AKRS3}
Ashkenazi-Golan G., Flesch J., and Solan E. (2022) 
Absorbing Blackwell games. arXiv:2208.11425.

\bibitem{Blackwell69}
Blackwell D. (1969)
Infinite $G_\delta$-games with imperfect information.
\emph{Applicationes Mathematicae}, \textbf{10}, 99--101.

\bibitem{BlackwellDiaconis1996}
Blackwell D. and Diaconis P. (1996) A non-measurable tail set. \emph{Lecture Notes Monograph Series} \textbf{30}, 1–5.

\bibitem{Bruyere2021}
Bruy\'{e}re V. (2021) Synthesis of equilibria in infinite-duration: games on graphs.
\emph{ACM SIGLOG News 8.2}, 4–29.

\bibitem{ChatterjeeHenzinger2012}
Chatterjee K. and Henzinger T.A. (2012) A survey of stochastic $\omega$-regular
games. \emph{Journal of Computer System Sciences}, \textbf{78}, 394–413.

\bibitem{Duggan2012}
Duggan J. (2012) Noisy stochastic games. \emph{Econometrica}, \textbf{80}, 2017--2045.

\bibitem{FilarVrieze}
Filar J. and Vrieze O. (2012)
\emph{Competitive Markov Decision Processes}.
Springer Science \& Business Media.

\bibitem{Fink1964}
Fink M.A. (1964)
Equilibrium in a stochastic $n$-person game. 
\emph{Journal of Science of the Hiroshima University}, Series AI (Mathematics), \textbf{28}, 89–93.

\bibitem{Flesch21}
Flesch J., Herings P.J.-J., Maes J., and Predtetchinski A. (2021)
Subgame maxmin strategies in zero-sum stochastic games with tolerance levels.
\emph{Dynamic Games and Applications}, \textbf{11}, 704–737.

\bibitem{6authors}
Flesch J., Kuipers J., Mashiah-Yaakovi A., Schoenmakers G.,
Shmaya E., Solan E., and Vrieze K. (2014) Non-existence of subgame-perfect $\ep$-equilibrium in perfect information games with infinite horizon. \emph{International Journal of Game Theory}, \textbf{43}, 945–951.

\bibitem{FleschSolan22a}
Flesch J. and Solan E. (2022a)
Equilibrium in two-player stochastic games with shift-invariant payoffs.
arXiv:2203.14492.

\bibitem{FleschSolan22b}
Flesch J. and Solan E. (2022b)
Repeated games with tail-measurable payoffs.
arXiv:2205.12164.


\bibitem{HornGimbert2008}
Horn F. and Gimbert H. (2008) Optimal strategies in perfect-information stochastic games with tail winning conditions. In: CoRR abs/0811.3978 (2008). arXiv: 0811.3978. 

\bibitem{JaskieziczNowak2018}
Ja\'{s}kiewicz and Nowak A.S. (2018) Non-zero-sum stochastic games. 
\emph{Handbook of Dynamic Game Theory}, pp. 281--344.

\bibitem{Levy2013}
Levy Y.J. (2013)
Discounted stochastic games with no stationary Nash equilibrium: two examples.
\emph{Econometrica}, \textbf{81}(5), 1973--2007.

\bibitem{LevyMcLennan2015}
Levy Y.J. and McLennan A. (2015) Corrigendum to ``Discounted stochastic games with no stationary Nash equilibrium: two examples''. \emph{Econometrica}, \textbf{83}, 1237--1252.

\bibitem{LevySolan2020}
Levy Y.J. and Solan E. (2020) Stochastic games. \emph{Complex Social and
Behavioral Systems: Game Theory and Agent-Based Models}, pp. 229--250.

\bibitem{Maitra93}
Maitra A. and Sudderth W. (1993)
Borel stochastic games with lim sup payoff. 
\emph{The Annals of Probability}, \textbf{21}, 861--885.

\bibitem{MaitraSudderth1998}
Maitra A. and Sudderth W. (1998) Finitely additive stochastic games with Borel measurable payoffs. \emph{International Journal of Game Theory} \textbf{27}, 257–267.


\bibitem{Martin98}
Martin D.A. (1998)
The determinacy of Blackwell games,
\emph{The Journal of Symbolic Logic}, \textbf{63}, 1565--1581.

\bibitem{Mashiah15}
Mashiah-Yaakovi A. (2015) 
Correlated equilibria in stochastic games with Borel measurable payoffs. 
\emph{Dynamic Games and Applications}, \textbf{5}, 120--135.

\bibitem{MertensNeyman1981}
  {Mertens J.-F. and Neyman A. (1981)
   Stochastic games,
   {\it International Journal of Game Theory}, {\bf 10}, 53--66.}

\bibitem{MertensParthasarathy1987}
Mertens J.-F. and  Parthasarathy T. (1987)
Equilibria for discounted stochastic games. In \emph{Stochastic Games and Applications}, by Neyman A. and Sorin S. (Eds.), 131--172. Springer, Dordrecht.

\bibitem{Neyman03}
Neyman A. (2003)
Stochastic games: existence of the minmax,
In \emph{Stochastic Games and Applications}, by Neyman A. and Sorin S. (Eds.), 
173--193. Springer, Dordrecht.

\bibitem{Nowak1985}
Nowak A. S. (1985). Existence of equilibrium stationary strategies in discounted noncooperative stochastic games with uncountable state space. \emph{Journal of Optimization Theory and Applications}, \textbf{45}, 591--602.

\bibitem{RenaultZiliotto2020}
Renault J. and Ziliotto B. (2020). Limit equilibrium payoffs in stochastic Games. \emph{Mathematics of Operations Research}, \textbf{45}, 889--895.

\bibitem{Rosenthal}
Rosenthal J. (1975) Nonmeasurable invariant sets. \emph{The American Mathematical Monthly}, \textbf{82}(5), 488–491.

\bibitem{Shapley1953}
Shapley L.S. (1953) Stochastic games. \emph{Proceedings of the National Academy of
Sciences of the United States of America}, \textbf{39}(10), 1095–1100.

\bibitem{Shmaya2011}
Shmaya E. (2011) The determinacy of infinite games with eventual perfect monitoring. \emph{Proceedings of the American Mathematical Society}, \textbf{139}, 3665--3678.

\bibitem{Simon2007}
Simon R.S. (2007) The structure of non-zero-sum stochastic games. \emph{Advances in Applied Mathematics}, \textbf{38}, 1--26.

\bibitem{Simon2016}
Simon R.S. (2016)
The challenge of non-zero-sum stochastic games. 
\emph{International Journal of Game Theory}, \textbf{45}, 191--204.

\bibitem{Solan2001}
Solan E. (2001)
Characterization of correlated equilibria in stochastic games. \emph{International Journal of Game Theory}, \textbf{30}(2), 259--277.

\bibitem{Solan2018}
Solan E. (2018)
Acceptable strategy profiles in stochastic games.
\emph{Games and Economic Behavior}, \textbf{108}, 523--540.

\bibitem{Solan2022}
Solan E. (2022)
\emph{Stochastic Games}.
Cambridge University Press.

\bibitem{SolanVieille2002}
Solan E. and Vieille N. (2002) Correlated equilibrium in stochastic games. \emph{Games and Economic Behavior}, \textbf{38}, 362--399.

\bibitem{SolanVieille2015}
Solan E. and Vieille N. (2015) Stochastic games. \emph{Proceedings of the National Academy of Sciences of the United States of America}, \textbf{112}(45), 13743--13746.

\bibitem{SorinVigeral2013}
Sorin S. and Vigeral G. (2013) Existence of the limit value of two person zero-sum discounted repeated games via comparison theorems. \emph{Journal of Optimization Theory and Applications}, \textbf{157}, 564--576.

\bibitem{Takahashi1964}
Takahashi M. (1964)
Equilibrium points of stochastic non-cooperative $n$-person games. \emph{Journal of Science of the Hiroshima University}, Series AI (Mathematics),
\textbf{28}, 95–99.

\bibitem{TV91}
Thuijsman F. and Vrieze O.J. (1991) Easy initial states in stochastic games. 
In \emph{Stochastic games and related topics} (pp. 85--100). Springer, Dordrecht.

\bibitem{Venel2015}
Venel X. (2015). Commutative stochastic games. \emph{Mathematics of Operations Research}, \textbf{40}, 403--428.

\bibitem{Vieille-a} 
Vieille N. (2000a) Two-player stochastic games I: A reduction.
\emph{Israel Journal of Mathematics}, \textbf{119}, 55--91.

\bibitem{Vieille-b}
Vieille N. (2000b): Two-player stochastic games II: The case of recursive games.
\emph{Israel Journal of Mathematics}, \textbf{119}, 93--126.

\bibitem{Vieille00c}
Vieille N. (2000c) 
Solvable states in $N$-player stochastic games. 
\emph{SIAM Journal on Control and Optimization}, \textbf{38}, 1794--1804.

\bibitem{VriezeThuijsman1989}
Vrieze O.J. and Thuijsman F. (1989) On equilibria in repeated games with absorbing states. \emph{International Journal of Game Theory}, \textbf{18}, 293--310.

\end{thebibliography}
\end{document}